\documentclass{amsart}
\usepackage[T1]{fontenc} 

\usepackage{amssymb}
\usepackage{amsthm}
\usepackage{amsmath}
\usepackage{xcolor}
\usepackage[matrix, arrow]{xy}
\xyoption{arrow}
\theoremstyle{plain}\newtheorem{Theorem}{Theorem}[section]
\theoremstyle{plain}
\theoremstyle{plain}\newtheorem{Corollary}[Theorem]{Corollary}
\theoremstyle{plain}\newtheorem{Lemma}[Theorem]{Lemma}
\theoremstyle{plain}\newtheorem{Proposition}[Theorem]{Proposition}
\theoremstyle{definition}\newtheorem{Definition}[Theorem]{Definition}
\theoremstyle{definition}\newtheorem{Example}[Theorem]{Example}
\theoremstyle{definition}
\theoremstyle{definition}
\theoremstyle{definition}\newtheorem{Notation}[Theorem]{Notation}
\theoremstyle{definition}\newtheorem{Remark}[Theorem]{Remark}
\theoremstyle{definition}

\def\CA{{\mathcal{A}}}  
  
\def\CC{{\mathcal{C}}}  
\def\CD{{\mathcal{D}}}  
\def\CE{{\mathcal{E}}}

\def\CT{{\mathcal{T}}}

\def\CA{{\mathcal{A}}}
\def\CA{{\mathcal{A}}}

\def\add{\mathrm{add}}
\def\ann{\mathrm{ann}}
               \def\tenk{\otimes_k}     
                 \def\ten{\otimes}

\def\chr{\mathrm{char}}   
\def\coker{\mathrm{coker}}  

\def\dim{\mathrm{dim}}

\def\End{\mathrm{End}}         
\def\Endbar{\underline{\mathrm{End}}}    
\def\Ext{\mathrm{Ext}}        

\def\Gr{\mathrm{Gr}}
\def\Hom{\mathrm{Hom}}           

\def\Hombar{\underline{\mathrm{Hom}}}

\def\ker{\mathrm{ker}}           

\def\Id{\mathrm{Id}} 
      \def\tenA{\otimes_A}
\def\Im{\mathrm{Im}}      
    
    \def\tenD{\otimes_D}

\def\mod{\mathrm{mod}}

\def\modbar{\underline{\mathrm{mod}}}

\def\uZ{\underline{Z}}

\def\op{\mathrm{op}}

  \def\uPsi{\underline{\Psi}}        
        
  \def\pr{\mathrm{pr}}

\def\proj{\mathrm{proj}}

\def\Res{\mathrm{Res}}

\def\soc{\mathrm{soc}}

\def\Tr{\mathrm{Tr}}

\makeindex
\title{On abelian subcategories of triangulated categories} 
\author{Markus Linckelmann} 
\date{\today}

\begin{document}

\begin{abstract}
The stable module category of a selfinjective algebra is triangulated,
but need not have any nontrivial $t$-structures, and in particular, full abelian
subcategories need not arise as hearts of a $t$-structure. 
The purpose of this paper is to investigate full abelian subcategories 
of triangulated categories whose exact structures are related, and
more precisely,  to explore relations between invariants of finite-dimensional 
selfinjective algebras and full abelian subcategories of their
stable module categories. 

\end{abstract}

\maketitle

\tableofcontents

\section{Introduction}

\begin{Definition} \label{distDef}
Let $\CC$ be a triangulated category with shift functor $\Sigma$.
A {\it distinguished abelian subcategory of } $\CC$ is a full additive 
subcategory $\CD$ of $\CC$ which is abelian, such that for any short 
exact sequence
$$\xymatrix{0 \ar[r] & X\ar[r]^{f} & Y \ar[r]^{g} &Z\ar[r] & 0}$$
in $\CD$ there exists a morphism $h : Z\to$ $\Sigma(X)$ in $\CC$ such 
that the triangle
$$\xymatrix{X\ar[r]^{f} & Y \ar[r]^{g} &Z\ar[r]^{h} & \Sigma(X)}$$
is exact in $\CC$. 

\end{Definition}

Proper abelian subcategories, introduced in \cite[Def. 1.2]{Jorg20}, 
admissible abelian subcategories, from \cite[Def. 1.2.5]{BBD}, and thus 
hearts of $t$-structures are distinguished abelian subcategories, but not all
distinguished abelian subcategories are proper abelian subcategories.

The main motivation for considering  this definition is the abundance of
distinguished abelian categories in stable module categories of 
finite-dimensional selfinjective algebras, and the hope that these may
therefore shed light on the invariants of selfinjective algebras in terms
of their stable module categories.  
A conjecture of Auslander-Reiten predicts that for $A$ a finite-dimensional
algebra over a field, the stable category $\modbar(A)$  of finitely
generated left $A$-modules 
should determine the number of isomorphism classes of nonprojective 
simple $A$-modules. If this conjecture were true for blocks of finite group
algebras, it would imply some cases of Alperin's weight conjecture.
By a result of Martinez-Villa \cite{MV}, it would suffice 
to prove the Auslander-Reiten  conjecture for selfinjective algebras.
If $A$ is selfinjective,  then $\modbar(A)$ is  triangulated.
The following result recasts the Auslander-Reiten conjecture for
selfinjective algebras in terms
of maximal distinguished abelian subcategories of $\modbar(A)$. 
For $\CD$ an abelian category, we denote by $\ell(\CD)$ the number of 
isomorphism classes of simple objects, with the convention $\ell(\CD)=$ 
$\infty$ if $\CD$ has infinitely many isomorphism classes of simple objects. 
For $A$ a finite-dimensional algebra over a field, we write $\ell(A)=$ 
$\ell(\mod(A))$; that is, $\ell(A)$ is
the number of isomorphism classes of simple $A$-modules. 

\begin{Theorem} \label{ellA}
Let $A$ be a finite-dimensional selfinjective algebra over a field such 
that all simple $A$-modules are nonprojective. The following hold.

\begin{enumerate}
\item[{\rm (i)}] If $\CD$ is a distinguished abelian subcategory of 
$\modbar(A)$ containing all simple $A$-modules, then the simple
$A$-modules are exactly the simple objects in $\CD$. In particular,
in that case we have $\ell(A)=\ell(\CD)$. 

\item[{\rm (ii)}]  The stable module
category  $\modbar(A)$ has a maximal distinguished abelian subcategory 
$\CD$ satisfying $\ell(\CD)=$ $\ell(A)$. 
\end{enumerate}
\end{Theorem}

Statement (i) of this theorem is Theorem \ref{AmodIembeddings}, and
statement (ii) will be proved in Section \ref{exactSection}. 
Note that $\modbar(A)$ may have distinguished abelian subcategories
$\CD$ satisfying $\ell(\CD)=\infty$; see Example 
\ref{ellDinfinite-example}.  
A subcategory $\CD$ of a triangulated catergory $\CC$ is called
{\it extension closed} if for any exact triangle $U\to V\to W\to \Sigma(U)$
in $\CC$ with $U$, $W$ belonging to $\CD$, the object $V$ is isomorphic 
to an object in $\CD$.
Hearts of $t$-structures on a triangulated category $\CC$ are extension 
closed distinguished abelian subcategories. In general, distinguished 
abelian subcategories need not be extensions closed.
The following result is a basic  construction principle for distinguished 
abelian subcategories in $\modbar(A)$, together with a sufficient
criterion for detecting certain distinguished abelian subcategories which are 
not extension closed.

\begin{Theorem} \label{AmodIembedding-intro}
Let $A$ be a finite-dimensional selfinjective algebra over a field
and let $I$ be a proper ideal in $A$. Denote by $r(I)$ the right 
annihilator of $I$ in $A$.

\begin{enumerate}
\item[{\rm (i)}]
The canonical map $A\to A/I$ induces an embedding $\mod(A/I)\to$ 
$\modbar(A)$ as a distinguished abelian subcategory in $\modbar(A)$
if and only if  $r(I)\subseteq$ $I$.  

\item[{\rm (ii)}]
Suppose that $r(I)\subseteq$ $I\subseteq$ $J(A)$ and that $A/I$ is 
selfinjective. Then  the distinguished abelian subcategory $\mod(A/I)$ 
of $\modbar(A)$ is not extension closed in $\modbar(A)$.
\end{enumerate}
\end{Theorem}

Statement (i)  will be proved as part of Theorem \ref{AmodIembedding1}, 
itself a consequence of the more general Theorem \ref{full-embeddings}, 
and statement (ii) follows from Proposition \ref{notextclosed}.
The first part of this theorem points to the fact that distinguished 
abelian subcategories of $\modbar(A)$ tend to come in varieties - see 
Proposition  \ref{AmodIembedding2}.

 If a triangulated category $\CC$ carries a 
structure of a monoidal category and if $\CD$ is a distinguished
abelian subcategory of $\CC$ which is also a monoidal
subcategory of $\CC$, we call $\CD$ a {\it monoidal distinguished
abelian subcategory of} $\CC$.
Stable module categories of finite group algebras provide examples
of monoidal distinguished abelian subcategories which do not arise as heart of 
a $t$-structure, and which are not extension closed. 
For $G$ a finite group and $p$ a prime, we denote by $O^p(G)$ the 
smallest normal subgroup of $G$ such that $G/O^p(G)$ is a $p$-group. 

\begin{Theorem} \label{kGmodN}
Let $k$ be a field of prime characteristic $p$ and $G$ a finite group. 
Let $N$ be a normal subgroup of $G$ of order divisible by $p$. 

\begin{enumerate}
\item[{\rm (i)}]
Restriction along the canonical surjection $G\to$ $G/N$ induces a 
full embedding of $\mod(kG/N)$ as a 
symmetric monoidal distinguished abelian subcategory in $\modbar(kG)$.

\item[{\rm (ii)}]
There is no nontrivial $t$-structure on $\modbar(kG)$; that is, the heart of 
any $t$-structure on $\modbar(kG)$ is zero. In particular, $\mod(kG/N)$ is
not the heart of a $t$-structure on $\modbar(kG)$. 

\item[{\rm (iii)}]
If $O^p(N)$ is a proper subgroup of $N$, then the abelian subcategory
$\mod(kG/N)$ of $\modbar(kG)$ is not extension closed in $\modbar(kG)$.
\end{enumerate}
\end{Theorem}

See Theorem \ref{kGexample}, Corollary \ref{sym-not}, Corollary 
\ref{kG-sym-not}, and Proposition  \ref{notextclosedProp} for more precise 
statements and proofs.
For $k$ a field of prime
characteristic $p$ and $P$ a finite $p$-group, the Auslander-Reiten 
conjecture is known to hold for $kP$ (cf. \cite[Theorem 3.4]{Listable}) . 
We use this to classify the distinguished
abelian subcategories of $\modbar(kP)$ which are equivalent to the
module categories of split finite-dimensional algebras in Theorem 
\ref{Pgroupdist}.

Section \ref{stablemoduleSection} describes some construction principles 
of distinguished abelian subcategories of stable module categories. 
Section \ref{simplestablemoduleSection} describes distinguished abelian
subcategories of $\modbar(A)$ whose simple objects are the
simple $A$-modules, where $A$ is a finite-dimensional selfinjective
algebra. Section  \ref{groupSection} specialises previous results to 
distinguished abelian subcategories in finite group algebras over a field 
of prime characteristic $p$, and includes a proof of the first statement
of Theorem \ref{kGmodN}. 
The section \ref{basicpropertySection} contains some general facts
on distinguished abelian subcategories. In particular, it is shown in 
Proposition \ref{hunique} that the morphism $h$ in Definition
\ref{distDef} is unique. Section \ref{exactSection} contains technicalities,
needed for the proof of Theorem \ref{ellA}, 
on the interplay between short exact sequences in $\mod(A)$ and
short exact sequences in a distinguished abelian subcategory $\CD$
of $\modbar(A)$.
The main result of Section \ref{ExtSection}  is a criterion on extension
closure of distinguished abelian subcategories, needed for the last part
of Theorem \ref{kGmodN}.
 Section \ref{CabSection2} relates embeddings of
module categories of selfinjective algebras to a result of Cabanes.
Section \ref{ExamplesSection} contains
examples and further remarks.

\begin{Remark} \label{motivationRemark}
The present paper, investigating  abelian subcategories of 
triangulated categories in situations where there are no nontrivial 
$t$-structures, started out as a speculation about  a possible
analogue of stability spaces (cf. \cite{Bridgeland}) for stable 
module categories of finite-dimensional selfinjective algebras. 
Another  interesting angle to pursue would be 
connections with abelian {\it quotient} categories of triangulated 
categories, which appear in numerous sources, for instance, in \cite{KZ}, 
\cite{IY}, \cite{GJ}, in the context of torsion and mutation pairs in 
triangulated categories. See also \cite{Dugas15}, which explores this
topic with a particular emphasis on stable module categories of 
finite-dimensional selfinjective algebras.
\end{Remark}

\begin{Notation}
Throughout this paper, $k$ is a field.
Modules are unital left modules, unless stated otherwise, and algebras
are nonzero unital associative.  Let $A$ be a
finite-dimensional $k$-algebra. We denote by $\mod(A)$ the abelian 
category of finitely generated $A$-modules. We denote by $\modbar(A)$ 
the stable module category of $\mod(A)$. That is, the objects of $\modbar(A)$
are the same as in $\mod(A)$, and for any two finitely generated 
$A$-modules $U$, $V$, the morphism space in $\modbar(A)$ from $U$ to $V$ 
is the $k$-space $\Hombar_A(U,V)=$ $\Hom_A(U,V)/\Hom_A^\pr(U,V)$, where 
$\Hom_A^\pr(U,V)$ is the space of all $A$-homomorphisms from $U$ to $V$ 
which factor through a projective $A$-module. Composition in 
$\modbar(A)$ is induced by that in $\mod(A)$. We write $\Endbar_A(U)=$ 
$\Hombar_A(U,U)$ and $\End^\pr_A(U)=$ $\Hom_A^\pr(U,U)$. See
\cite[\S 2.13]{LiBookI} for more details. The {\it Nakayama functor of} $A$ is
the functor $\nu=$ $A^\vee\tenA-$ on $\mod(A)$, where $A^\vee=$ 
$\Hom_k(A,k)$ is the $k$-dual of $A$ regarded as an $A$-$A$-bimodule. 

The algebra $A$ is called {\it selfinjective} if $A$ is injective as a 
left (or right) $A$-module. Equivalently, $A$ is selfinjective if the 
classes of finitely generated projective and injective $A$-modules 
coincide. By results of Happel \cite{Hap}, if $A$ is selfinjective, 
then $\modbar(A)$ is a triangulated category, with shift functor 
$\Sigma$ induced by the operator sending an $A$-module $U$  to the cokernel 
of an injective envelope $U\to$ $I_U$, and with exact triangles in 
$\modbar(A)$ induced by short exact sequences in $\mod(A)$. 
If $A$ is selfinjective, then the Nakayama functor $\nu$ on $\mod(A)$
is an equivalence and induces an equivalence on $\modbar(A)$, and
$\tau=$ $\Sigma^{-2}\circ\nu$ is the Auslander-Reiten translate. 
The algebra $A$ is a {\it Frobenius algebra} if $A$ is isomorphic to its
$k$-dual $A^\vee$ as a left or right $A$-module. A Frobenius algebra
is selfinjective. The algebra $A$ is called {\it symmetric} if $A$ is
isomorphic to its $k$-dual $A^\vee$ as an $A$-$A$-bimodule. The image
$s$ in $A^\vee$ of $1_A$ under some bimodule isomorphism $A\cong$ 
$A^\vee$ is called a {\it symmetrising form of} $A$. If $A$ is symmetric,
then $A$ is a Frobenius algebra, hence selfinjective, and the Nakayama
functor is isomorphic to the identity functor on $\mod(A)$. Finite group 
algebras, their blocks, and Iwahori-Hecke algebras are symmetric.
See for instance \cite[Ch. III]{SkYaI}, \cite[\S\S 2.11, 2.14]{LiBookI} and 
\cite[Appendix A.3]{LiBookII} for more background. 

For $I$ a left ideal in $A$, its right annihilator
$r(I)=$ $\{a\in A\ |\ Ia=0\}$ is a right ideal, and for $J$ a right
ideal in $A$, its left annihilator $l(J)=$ $\{a\in A\ |\ aJ=0\}$ is
a left ideal. If $I$ is an ideal in $A$, then so are $r(I)$ and $l(I)$.
By results of Nakayama in \cite{NakI}, \cite{NakII}, if $A$ is 
selfinjective, then the correspondencs $I\mapsto r(I)$ and 
$J\mapsto l(J)$ are inclusion 
reversing bijections between the sets of left and right ideals in $A$.
These bijections are inverse to each other and restrict to bijections
on the set of ideals in $A$. In particular, for any ideal $I$ in $A$ 
we have $r(l(I))=$ $I=$ $l(r(I))$. If $A$ is a Frobenius algebra,
then $\dim_k(I)+\dim_k(r(I))=$ $\dim_k(A)$, and if $A$ is symmetric, 
then  $r(I)=l(I)$ for any ideal $I$ in $A$. See 
\cite[Chapter IV, Section 6]{SkYaI} for details.

We will make use without further comment of the standard 
Tensor-Hom adjunction.
\end{Notation}

\section{Distinguished abelian subcategories in stable module
categories} \label{stablemoduleSection}

The stable module category of a finite-dimensional non-semisimple
selfinjective $k$-algebra $A$ need not have any 
$t$-structures (see Proposition \ref{selfinj-not} and Corollary 
\ref{sym-not}), but it  always has  distinguished abelian subcategories, 
and these tend to come in  varieties (see Proposition 
\ref{AmodIembedding2}). 
The first result in this section describes those distinguished abelian 
subcategories of $\modbar(A)$  which arise as image of a full abelian 
subcategory of $\mod(A)$ equivalent to $\mod(D)$, for some other
finite-dimensional $k$-algebra $D$. 

\begin{Theorem} \label{full-embeddings}
Let $A$ be a finite-dimensional selfinjective $k$-algebra, and let
$D$ be a finite-dimensional $k$-algebra. Let $Y$ be a finitely generated
$A$-$D$-bimodule. 
The functor $Y\tenD-$ is a full exact embedding of $\mod(D)$ into
$\mod(A)$ and induces an embedding of $\mod(D)$ as a distinguished 
abelian subcategory of $\modbar(A)$ if and only if the
following conditions hold. 
\begin{enumerate}
\item
$\End_A^\pr(Y) = \{0\}$,
\item
$Y$ is projective as a right $D$-module, and
\item
the Tensor-Hom adjunction unit maps $V \to \Hom_A(Y,Y\tenD V)$,
$v\mapsto (y\mapsto y\ten v))$, 
are isomorphisms, for all finitely generated $D$-modules $V$.
\end{enumerate}
\end{Theorem}

We state some parts of the proof of Theorem \ref{full-embeddings} as 
separate lemmas in slightly greater generality. 

\begin{Lemma} \label{detectLemma}
Let $A$ be a finite-dimensional selfinjective $k$-algebra. Let 
$\CD$ be a full abelian subcategory of $\mod(A)$ such that
$\Hom_A^\pr(U,V)=\{0\}$ for all $A$-modules $U$, $V$ in $\CD$.
Then the image of $\CD$ in $\modbar(A)$ is a distinguished abelian
subcategory of $\modbar(A)$, which as an abelian category, is
equivalent to $\CD$.
\end{Lemma}

\begin{proof}
The fact that $\CD$ is a full subcategory of $\mod(A)$, together
with the hypothesis $\Hom_A^\pr(U,V)=\{0\}$ for all $U$, $V$ in $\CD$,
implies that the image of $\CD$ in $\modbar(A)$ is a full subcategory
of $\modbar(A)$ which is equivalent to $\CD$. By the assumptions on
$\CD$, exact sequences in $\CD$ remain exact in $\mod(A)$. 
Since distinguished triangles in $\modbar(A)$ are induced by short 
exact sequences in $\mod(A)$, it follows that the image of $\CD$ in 
$\modbar(A)$ is a distinguished abelian subcategory.
\end{proof}

\begin{Lemma} \label{Homprzero1}
Let $A$ be a finite-dimensional $k$-algebra and $Y$ a finitely
generated $A$-module such that $\End_A^\pr(Y)=$ $\{0\}$. Set $D=$ 
$\End_A(Y)^\op$. The following hold.

\begin{enumerate}
\item[{\rm (i)}]
Let $m$, $n$ be positive integers and let $U$, $V$ be quotients of the
$A$-modules $Y^m$, $Y^n$, respectively. Then $\Hom_A^\pr(U,V)=$
$\{0\}$. 

\item[{\rm (ii)}]
For any two finitely generated $D$-modules $M$, $N$ we have
$\Hom_A^\pr(Y\ten_{D} M, Y\ten_{D} N) = \{0\}$.
\end{enumerate}
\end{Lemma}

\begin{proof}
With the assumptions in (i), there are surjective 
$A$-homomorphisms $\alpha : Y^m\to$ $U$ and $\beta : Y^n\to$ $V$. Let 
$\psi : U\to$ $V$ be an $A$-homomorphism which factors through a 
projective $A$-module $P$. Let $\gamma : U\to$ $P$ and $\delta : P\to$ 
$V$ be $A$-homomorphism such that $\psi =$ $\delta\circ\gamma$. Since 
$P$ is projective and $\beta$ is surjective, there is an 
$A$-homomorphism $\epsilon : P\to$ $Y^n$ such that $\beta\circ\epsilon=$ 
$\delta$. Note that the homomorphism 
$\epsilon\circ\gamma\circ\alpha : Y^m\to$ $Y^n$ factors through $P$, 
hence is zero by the assumptions on $Y$. Thus $\psi\circ\alpha=$ 
$\delta\circ\gamma\circ\alpha=$ 
$\beta\circ\epsilon\circ\gamma\circ\alpha=$ $0$, and hence $\psi=0$ as 
$\alpha$ is surjective. This shows that $\Hom^\pr_A(U,V)=$ $\{0\}$ as 
stated. Let $M$, $N$ be as in (ii). As a left $A$-module, we have 
$Y\tenk M\cong$ $Y^m$, where $m=$ $\dim_k(M)$, and we have a canonical 
surjection of $A$-modules $Y\tenk M\to$ $Y\ten_{D} M$; similarly for 
$Y\tenk N$. Thus (ii) is a special case of (i). 
\end{proof}

The following observation is well-known (see the papers \cite{AlperinStatic} 
and \cite{NaumanStatic1} on static and adstatic modules). We include a 
short proof for convenience.

\begin{Lemma} \label{embedLemma1}
Let $A$, $D$ be finite-dimensional $k$-algebras, and let $Y$ be a 
finitely generated $A$-$D$-bimodule. 
The functor $Y\ten_D - : \mod(D)\to$ $\mod(A)$ is a full $k$-linear 
embedding  if and only if the adjunction unit $V\to$ 
$\Hom_A(Y,Y\tenD V)$, $v\mapsto (y\mapsto y\ten v)$ is an isomorphism, 
for every finitely generated $D$-module $V$.
\end{Lemma}

\begin{proof}
The functor $Y\tenD-$ is a full embedding 
if and only if for any two finitely generated $D$-modules $U$, $V$, the map 
$\Hom_D(U,V) \to \Hom_A(Y\tenD U, Y\tenD V)$
induced by $Y\tenD-$ is an isomorphism, hence if and only if the canonical
map $\Hom_D(U,V)\to$ $\Hom_D(U,\Hom_A(Y,Y\tenD V))$ is an isomorphism.
 By considering the case $U=D$, one sees that this is the
case if and only if the adjunction map $V\to\Hom_A(Y,Y\tenD V)$ itself
is an isomorphism, whence the result.
\end{proof}

\begin{proof}[{Proof of Theorem \ref{full-embeddings}}]
By Lemma \ref{embedLemma1}, the functor 
$Y\tenD- : \mod(D)\to \mod(A)$ is a full embedding, if and only if 
the condition (3) holds. This embedding is exact if and only
if $Y$ is flat as a right $D$-module. Since $Y$ is finitely generated,
this is equivalent with requiring condition (2). 
It follows from the Lemmas \ref{detectLemma} and \ref{Homprzero1}
that the composition with the canonical functor $\mod(A)\to$ 
$\modbar(A)$ yields an embedding of $\mod(D)$ as a distinguished
abelian subcategory in $\modbar(A)$ if and only if (1) holds as well.
This concludes the proof of Theorem \ref{full-embeddings}.
\end{proof}

If both $A$ and $D$ are selfinjective, then Theorem 
\ref{full-embeddings} yields the following result.

\begin{Theorem} \label{modDmodA}
Let $A$ be a finite-dimensional selfinjective $k$-algebra. 
Let $Y$ be a finitely generated $A$-module. Suppose that
$\End_A(Y)$ is selfinjective. Set $D=$ $\End_A(Y)^\op$, and regard
$Y$ as an $A$-$D$-bimodule. The following are equivalent.

\begin{enumerate}
\item[{\rm (i)}]
The functor $Y\tenD- : \mod(D)\to$ $\mod(A)$ is a full exact embedding 
and induces an embedding of $\mod(D)$ as a distinguished abelian 
subcategory in $\modbar(A)$.

\item[{\rm (ii)}]
We have $\End^\pr_A(Y)=$ $\{0\}$, and $Y$ is projective as an
$\End_A(Y)$-module.
\end{enumerate}

\end{Theorem}

\begin{proof}
If (i) holds, then (ii) holds by Theorem \ref{full-embeddings}.
Suppose that (ii) holds. Note that the hypotheses imply that $D=$ 
$\End_A(Y)^\op$ is selfinjective and that $Y$ is projective as a right 
$D$-module.
Thus the conditions (1) and (2) in Theorem \ref{full-embeddings} are
satisfied. We need to show that condition (3) in that Theorem holds
as well. That is, given a finitely generated $D$-module $V$, we need
to show that the adjunction map $V\to$ $\Hom_A(Y,Y\tenD V)$ is an
isomorphism. Note that this is clear if $V=D$ as a consequence of
the assumption $D=$ $\End_A(Y)^\op$. Thus this is the case for $V$
any free $D$-module of finite rank. In general, since $D$ is 
selfinjective, $V$ is isomorphic to a submodule of a free $D$-module
of finite rank. Thus there is an exact sequence of $D$-modules of the
form
$$\xymatrix{0 \ar[r] & V\ar[r] & D^n \ar[r] & D^m }$$
for some positive integers $n$ and $m$. 
By the hypotheses in (ii), the functor $Y\tenD-$ is exact, and hence
we have an exact sequence of $A$-modules of the form
$$\xymatrix{0 \ar[r] & Y\tenD V\ar[r] & Y\tenD D^n \ar[r] 
& Y\tenD D^m }$$
Since the functor $\Hom_A(Y,-)$ is left exact, this yields an exact
sequence of $D$-modules of the form 
$$\xymatrix{0 \ar[r] & \Hom_A(Y, Y\tenD V)\ar[r] 
& \Hom_A(Y,Y\tenD D^n) \ar[r] & \Hom_A(Y, Y\tenD D^m) }$$
By naturality of the adjunction maps, we get a commutative diagram
of $D$-modules with exact rows
$$\xymatrix{0 \ar[r] & V\ar[r] \ar[d]_{\alpha} 
& D^n \ar[r] \ar[d]^{\beta} & D^m \ar[d]^{\gamma} \\ 
0 \ar[r] & \Hom_A(Y, Y\tenD V)\ar[r] 
& \Hom_A(Y,Y\tenD D^n) \ar[r] & \Hom_A(Y, Y\tenD D^m) }$$
where $\alpha$, $\beta$, $\gamma$ are the adjunction maps. By the
above remarks, $\beta$ and $\gamma$ are isomorphisms. The exactness
of the rows implies that $\alpha$ is an isomorphism as well.
This shows that condition (3) in Theorem \ref{full-embeddings} holds
as well, and hence the result follows from Theorem
\ref{full-embeddings}.
\end{proof}

Using a Theorem of Cabanes \cite[Theorem 2]{CabAst} one can identify 
the image of the functor $Y\tenD-$ in Theorem \ref{modDmodA}
more precisely; see Section \ref{CabSection2}. 

If $A$ is a Frobenius algebra over an algebraically closed field, then 
the ideals $I$ containing their right annihilators form subvarieties of 
certain Grassmannians.

\begin{Proposition} \label{AmodIembedding2}
Let $A$ be a finite-dimensional Frobenius algebra over $k$. Suppose that 
$k$ is algebraically closed. The set of proper ideals $I$ in $A$
satisfying $r(I)\subseteq$ $I$ is a projective variety whose connected 
components are subvarieties of the Grassmannians $\Gr(n,A)$, where 
$\frac{\dim_k(A)}{2}\leq n < \dim_k(A)$. 
\end{Proposition}

\begin{proof}
If $A$ is a Frobenius algebra, then
$\dim_k(A)=$ $\dim_k(I)+\dim_k(r(I))$. Since $r(I)\subseteq$ $I$, it
follows that $\dim_k(A)\leq$ $2\dim_k(I)$. Thus the ideals satisfying 
$r(I)\subseteq$ $I$ satisfy $\frac{\dim_k(A)}{2}\leq$ $\dim_k(I)$. In 
each dimension, they form subvarieties of the Grassmannians, since being 
an ideal with an annihilator of a fixed dimension is obviously a 
polynomial condition (obtained by fixing a $k$-basis of $A$).
\end{proof}

If $A$ is symmetric, then $A$ is selfinjective, but none of the 
distinguished abelian subcategories constructed above arises as the 
heart of a $t$-structure. More precisely, we have the following result.

\begin{Proposition} \label{selfinj-not}
Let $A$ be a finite-dimensional selfinjective $k$-algebra. Denote by
$\nu$ the Nakayama functor on $\modbar(A)$. Then $\modbar(A)$ 
has no nontrivial $t$-structure which is stable under $\nu$; that is, the 
heart of any $\nu$-stable $t$-structure  on $\modbar(A)$ is zero. 
\end{Proposition}

\begin{proof}
Let $(\CC^{\leq 0}, \CC^{\geq 0})$ be a $t$-structure on $\CC=$ 
$\modbar(A)$. Suppose that this $t$-structure is preserved by the
Nakayama functor $\nu$. For any $A$-module $U$ in $\CC^{\leq 0}$ and any 
$A$-module $V$ in $\CC^{\geq 0}$ we have $\Hombar_A(U,\Sigma^{-1}(V))=$ 
$\{0\}$. Auslander-Reiten duality for selfinjective algebras yields a duality  
between the space $\Hombar(U,\Sigma^{-1}(V))\cong$ 
$\Hombar_A(\nu(U), \nu(\Sigma^{-1}(V)))$ and
$\Hombar_A(V,\nu(U))$ (see e. g. \cite[Ch. III, Theorem 6.3]{SkYaI}). 
Thus $\Hombar_A(V,\nu(U))=$ $\{0\}$. Since 
the heart $\CC^{\leq 0}\cap\CC^{\geq 0}$ of 
the $t$-structure is $\nu$-stable, it follows that all morphisms in the
heart of this $t$-structure are zero. 
\end{proof}

\begin{Corollary} \label{sym-not}
Let $A$ be a finite-dimensional symmetric $k$-algebra. Then $\modbar(A)$ 
has no nontrivial $t$-structure; that is, the heart of any $t$-structure 
on $\modbar(A)$ is zero. 
\end{Corollary}

\begin{proof}
The Nakayama functor of a symmetric algebra is isomorphic to the identity
functor, and hence the statement is a special case of Proposition
\ref{selfinj-not}.
\end{proof}

Note that this Corollary follows also from more general results on negative Calabi-Yau
triangulated categories in \cite[\S 5.1]{HJY}, combined with Tate duality for symmetric
algebras (Tate duality for symmetric algebras  is the specialisation of the aforementioned 
Auslander-Reiten duality to the case where the Nakayama functor is isomorphic to the 
identity). 
Since finite group algebras are symmetric, this
implies in particular the second statement of Theorem \ref{kGmodN}:

\begin{Corollary} \label{kG-sym-not} 
Let $G$ be a finite group. Then $\modbar(kG)$ has no 
nontrivial $t$-structure; that is, the heart of any $t$-structure on
$\modbar(kG)$ is zero. 
\end{Corollary}

\begin{Remark} \label{Pic-embedding-Remark}
With the notation of Theorem \ref{full-embeddings}, not every
embedding of $\mod(D)$ as a distinguished abelian subcategory
of $\modbar(A)$ lifts in general to a full embedding $\mod(D)\to$
$\mod(A)$. Suppose that $Y\tenD- : \mod(D)\to$ $\mod(A)$ is a full 
exact embedding and induces an embedding $\mod(D)\to$ $\modbar(A)$ as 
distinguished abelian subcategory. Let $M$ be an $A$-$A$-bimodule 
inducing a stable equivalence of Morita type on $A$. Then the functor 
$M\tenA Y\tenD- : \mod(D) \to$ $\mod(A)$ is exact but no longer
necessarily full. It induces still an embedding of $\mod(D)$ as a 
distinguished abelian subcategory, because the functor $M\tenA-$
induces a triangulated equivalence on $\modbar(A)$, hence permutes
distinguished abelian subcategories.
It is not clear whether an embedding $\mod(D)\to$ $\modbar(A)$ as a 
distinguished abelian subcategory is necessarily induced by tensoring 
with a suitable $A$-$D$-bimodule. 
\end{Remark} 

\begin{Remark} \label{derived-embed}
Let $A$ be a finite-dimensional selfinjective $k$-algebra, $D$ a 
finite-dimensional $k$-algebra, and $Y$ an $A$-$D$-bimodule which is 
finitely generated projective as a right $D$-module. Then the functor 
$\mod(D)\to$ $\modbar(A)$ induced by $Y\tenD-$ extends to a
functor of triangulated categories $D^b(\mod(D))\to$ $\modbar(A)$.
Indeed, since $Y$ is finitely generated projective as a right $D$-module, 
it follows that $Y\tenD-$ induces a functor $D^b(\mod(D))\to$ 
$D^b(\mod(A))$. Composed with the canonical functor
$D^b(\mod(A))\to$ $\modbar(A)$ from \cite[Theorem 2.1]{Rick} or 
\cite[Theorem 4.4.1]{Buch}, this yields a 
functor $D^b(\mod(D))\to$ $\modbar(A)$ as stated. 
\end{Remark}

\section{Simple modules in distinguished abelian subcategories} 
\label{simplestablemoduleSection}

We consider in this section distinguished abelian subcategories of 
$\modbar(A)$ whose simple objects are simple $A$-modules, where
$A$ is a finite-dimensional selfinjective $k$-algebra.
The first theorem is essentially a special case of Theorem 
\ref{full-embeddings}, and it implies Theorem 
\ref{AmodIembedding-intro} (i), describing those distinguished
abelian subcategories which arise from quotients of $A$. 

\begin{Theorem} \label{AmodIembedding1}
Let $A$ be a finite-dimensional selfinjective $k$-algebra
and let $I$ be a proper ideal in $A$. The following statements are
equivalent.

\begin{enumerate}
\item[{\rm (i)}]
The composition of canonical functors $\mod(A/I)\to \mod(A)\to 
\modbar(A)$ is an embedding of $\mod(A/I)$ as a distinguished abelian 
subcategory in $\modbar(A)$.

\item[{\rm (ii)}]
The ideal $I$ contains its right annihilator $r(I)$. 

\item[{\rm (iii)}]
We have $\End^\pr_A(A/I)=\{0\}$. 

\item[{\rm (iv)}]
For any two finitely generated $A/I$-modules $U$, $V$, we have 
$\Hom^\pr_A(U,V)=$ $\{0\}$.

\item[{\rm (v)}]
We have $r(I)^2=\{0\}$.
\end{enumerate}
\end{Theorem}

Any full abelian subcategory of a distinguished abelian subcategory of a 
triangulated category is clearly again a distinguished abelian 
subcategory. In particular, if the canonical functor $\mod(A/I)\to$ 
$\modbar(A)$ is an embedding  as a distinguished abelian subcategory, 
then so is the canonical functor $\mod(A/J)\to$ $\modbar(A)$ for any 
ideal $J$ which contains $I$, because this factors through the 
embedding $\mod(A/J)\to$ $\mod(A/I)$ induced by the canonical 
surjection $A/I\to$ $A/J$.

Every ideal which squares to zero gives rise to a distinguished abelian
subcategory in the stable module category of a selfinjective algebra.

\begin{Corollary} \label{J2zeroCor}
Let $A$ be a finite-dimensional selfinjective $k$-algebra, and let $J$ be
an ideal in $A$ such that $J^2=\{0\}$. Set $I=$ $l(J)$. Then the canonical
surjection $A\to$ $A/I$ induces an embedding $\mod(A/I)\to$ $\modbar(A)$
of $\mod(A/I)$ as a distinguished abelian subcategory in $\modbar(A)$.
\end{Corollary}

\begin{proof}
We have $r(I)=$ $r(l(J))=J$, hence $r(I)^2=$ $\{0\}$ by the assumptions.
The result follows from the equivalence of the statements (i) and (v) in
Theorem \ref{AmodIembedding1}.
\end{proof}

If $A$ is a finite-dimensional Hopf algebra, then 
$\mod(A)$ is a monoidal abelian category, and $A$ is selfinjective, by 
a result of Larson and Sweedler \cite{LarSwe}. 
Given two $A$-modules $U$, $V$, if one of $U$, $V$ is projective, then 
so is $U\tenk V$ (see e. g. \cite[Proposition 3.1.5]{BenI}). Thus 
$\modbar(A)$ is a monoidal triangulated category. If $I$ is a proper 
Hopf ideal in $A$, then $A/I$ is a Hopf algebra, the canonical 
surjection $A\to A/I$ is a homomorphism of Hopf algebras, and hence 
induces a full embedding of monoidal categories $\mod(A/I)\to$ $\mod(A)$. 
Thus Theorem \ref{AmodIembedding1} implies immediately the following 
observation.

\begin{Corollary} \label{AmodIembeddingHopf}
Let $A$ be a finite-dimensional Hopf algebra over $k$ and let $I$ be a 
proper Hopf ideal in $A$ containing its right annihilator $r(I)$. Then 
the composition of canonical functors $\mod(A/I)\to \mod(A)\to 
\modbar(A)$ is an embedding of $\mod(A/I)$ as a monoidal distinguished 
abelian subcategory in the monoidal triangulated category $\modbar(A)$.
\end{Corollary}

Some of the implications in Theorem \ref{AmodIembedding1} hold in 
slightly greater generality. 

\begin{Lemma} \label{Endprzero}
Let $A$ be a finite-dimensional $k$-algebra, and let $J$ 
be a proper left ideal in $A$. We have $\End^\pr_A(A/J)=\{0\}$
if and only if $r(J)\subseteq$ $J$.
\end{Lemma}

\begin{proof}
Note first that if $\beta : A/J\to A$ is an $A$-homomorphism, then
$\beta\circ\pi$ is an $A$-endomorphism of $A$ with kernel containing
$J$, hence induced by right multiplication with an element $y\in$ $r(J)$.
Conversely, right multiplication with an element $y\in$ $r(J)$ factors
through $\pi$. Let $\alpha : A/J\to$ $A/J$ be an endomorphism
of $A/J$ as a left $A$-module such that $\alpha$ factors through
a projective $A$-module. Then $\alpha$ factors through the canonical
surjection $\pi : A\to$ $A/J$; that is, there is an $A$-homomorphism
$\beta : A/J\to$ $A$ such that $\alpha=$ $\pi\circ\beta$. By the above,
the endomorphism  $\beta\circ\pi$ of $A$ is induced by right 
multiplication with an element $y\in$ $r(J)$. Since $\pi$ is surjective,
we have $\alpha=0$ if and only if $\alpha\circ\pi=$ 
$\pi\circ\beta\circ\pi=0$, or equivalently, if and only if 
$\Im(\beta\circ\pi)\subseteq$ $\ker(\pi)=J$.  Since the image of 
$\beta\circ\pi$ is $Ay$, it follows that $\alpha=0$ if and only if 
$y\in$ $J$. The result follows.
\end{proof}

\begin{Lemma} \label{Homprzero}
Let $A$ be a finite-dimensional $k$-algebra and let $I$ be a proper
ideal in $A$. Suppose that $r(I)\subseteq$ $I$.
Then for any two $A/I$-modules $U$, $V$ we have 
$\Hom_A^\pr(U,V) = \{0\}\ .$
\end{Lemma}

\begin{proof}
Set $Y=$ $A/I$, regarded as an $A$-$A/I$-bimodule. 
Then $\End_A^\pr(Y)=\{0\}$ by Lemma \ref{Endprzero}, and we 
have $\End_A(Y)\cong$ $(A/I)^\op$, hence $D=\End_A(Y)^\op\cong$  $A/I$. 
Using this isomorphism, if $U$ is an $A/I$-module, then $Y\ten_{A/I} U=$ 
$A/I\ten_{A/I} U\cong$ $U$, regarded as an $A$-module via the canonical 
surjection $A\to A/I$. The result follows from Lemma \ref{Homprzero1}.
\end{proof}

\begin{Lemma} \label{IcontainsrI}
Let $A$ be a finite-dimensional selfinjective $k$-algebra, and let $I$ 
be an ideal in $A$. The following are equivalent.

\begin{enumerate}
\item[{\rm (i)}]
We have $r(I) \subseteq I$.
\item[{\rm (ii)}]
We have $l(I) \subseteq I$.
\item[{\rm (iii)}]
We have $r(I)^2=0$. 
\item[{\rm (iv)}]
We have $l(I)^2=0$.
\end{enumerate}
\end{Lemma}

\begin{proof}
If $r(I)\subseteq$ $I$, then taking left annihilators yields
$I=$ $l(r(I))\supseteq l(I)$, so (i) implies (ii). A similar argument
shows that (ii) implies (i). Since $I\cdot r(I)=0$, it follows that
if $r(I)\subseteq$ $I$, then $r(I)^2=0$. Thus (i) implies (iii). 
A similar argument shows that (ii) implies (iv). If $r(I)^2=0$, then
$r(I)\subseteq$ $l(r(I))=I$, so (iii) implies (i), and a similar
argument shows that (iv) implies (ii). 
\end{proof}

\begin{proof}[{Proof of Theorem \ref{AmodIembedding1}}]
We are going to prove Theorem \ref{AmodIembedding1} as a special case 
of Theorem \ref{full-embeddings}. Set $Y=$ $A/I$, regarded as an 
$A$-$A/I$-bimodule. 
We have $\End_A(Y)=$ $\End_{A/I}(Y)\cong$ $(A/I)^\op$. Clearly $Y$
is projective as a right $A/I$-module. Given an $A/I$-module $V$,
the adjunction unit $V\to$ $\Hom_A(A/I, A/I\ten_{A/I} V)$ is
trivially an isomorphism. Thus the $A$-$A/I$-bimodule satisfies the
conditions (2) and (3) in Theorem \ref{full-embeddings}. Therefore
the composition of functors $\mod(A/I)\to$ $\mod(A)\to$ $\modbar(A)$
is an embedding of $\mod(A/I)$ as a distinguished abelian subcategory
if and only if (1) holds, that is, if and only if $\End_A^\pr(A/I)=$
$\{0\}$. This proves the equivalence of (i) and (iii).  
It follows from Lemma \ref{Endprzero} that the statements (ii) and 
(iii) are equivalent. The implication (iii)  $\Rightarrow$ (iv) follows 
from Lemma \ref{Homprzero}. The implication (iv) $\Rightarrow$ (i) 
follows from Lemma \ref{detectLemma}. The equivalence of (ii) and (v) 
holds by Lemma \ref{IcontainsrI}. 
\end{proof}

\begin{Remark} \label{leftrightRemark}
Lemma \ref{IcontainsrI} implies that working with left or right modules 
yields equivalent statements. To illustrate this point, by Theorem
\ref{AmodIembedding1}, we have $r(I)\subseteq$ $I$ if and only if we
have a full embedding $\mod(A/I)\to$ $\modbar(A)$. There is an
obvious right module analogue which states that $l(I)\subseteq$ $I$ if 
and only if we have a full embedding $\mod((A/I)^\op)\to$ 
$\modbar(A^\op)$. Thus Lemma \ref{IcontainsrI} implies that we have a 
full embedding $\mod(A/I)\to$ $\modbar(A)$ if and only if we have a 
full embedding $\mod((A/I)^\op)\to$ $\modbar(A^\op)$. In other words, 
the full distinguished abelian subcategories in $\modbar(A)$ and
$\modbar(A^\op)$ constructed in Theorem \ref{AmodIembedding1} and its 
right module analogue correspond bijectively to each other.
\end{Remark}

\begin{Theorem} \label{AmodIembeddings}
Let $A$ be a finite-dimensional selfinjective $k$-algebra such that all 
simple $A$-modules are nonprojective. Let $\CD$ be a distinguished 
abelian subcategory of $\modbar(A)$. Suppose that $\CD$ contains all 
simple $A$-modules.  
The simple $A$-modules are exactly all simple objects in $\CD$. 
In particular, we have $\ell(\CD)=\ell(A)$.
\end{Theorem}

\begin{proof}
Let $U$ be an indecomposable nonprojective $A$-module belonging to
$\CD$, and let $S$ be a simple $A$-module. Let $\psi : U\to$ $S$ be an 
$A$-homomorphism, and denote by $\underline\psi$ the image of $\psi$ 
in $\Hombar_A(U,S)$. Note that $U$, $S$ are both nonzero objects 
$\CD$, by the assumptions.

We are going to show first that if $\psi$ is not an isomorphism in
$\mod(A)$, then $\underline\psi$ is not a monomorphism in $\CD$. If 
$\underline\psi$ is zero, there is nothing to show. If 
$\underline\psi$ is nonzero, then $\psi$ is surjective because $S$ is 
a simple $A$-module. Since $\psi$ is not an isomorphism we have 
 $\ker(\psi)$ is nonzero. Let $T$ be a simple $A$-submodule
of $\ker(\psi)$. The inclusion $T\to$ $U$ is an injective
$A$-homomorphism, hence its image in $\modbar(A)$ is a nonzero 
morphism in $\modbar(A)$. By construction, the composition
$T\to$ $U\to$ $S$ is zero in $\modbar(A)$. 
Since also $T$ belongs to $\CD$, it follows that $\underline\psi$
is not a monomorphism in $\CD$. 

This argument shows that $S$ is a simple object of $\CD$.
Indeed, if not, there would have to be a monomorphism $U\to$ $S$
in $\CD$ which is not an isomorphism. But by the first paragraph, 
any such monomorphism is inducd by an isomorphism in $\mod(A)$,
so is an isomorphism in $\CD$ as well.
This argument also shows that $\CD$ contains no other simple
objects. Indeed, let $U$ be an indecomposable nonprojective 
$A$-module which is a simple object in $\CD$.  Consider a surjective
$A$-homomorphism $\psi : U\to$ $S$ onto some simple $A$-module $S$.
Then $S$ belongs to $\CD$, and the image $\underline\psi$ is a
monomorphism in $\CD$ because $U$ is simple in $\CD$. But then 
$\psi$ is an isomorphism by the first argument.
Thus the simple $A$-modules are exactly the simple
objects in $\CD$, whence the result.
\end{proof}

\begin{Corollary} \label{AmodIembeddings-Cor1}
Let $A$ be a finite-dimensional selfinjective $k$-algebra such that all 
simple $A$-modules are nonprojective. Let $I$ be an ideal such that 
$r(I)\subseteq$ $I\subseteq$ $J(A)$. Let $\CD$ be a distinguished 
abelian subcategory of $\modbar(A)$ containing $\mod(A/I)$. Then the 
simple $A$-modules are exactly the simple objects in $\CD$.
\end{Corollary}

\begin{proof}
The hypothesis $r(I)\subseteq$ $I$ implies, by Theorem 
\ref{AmodIembedding1}, that $\mod(A/I)$ is a distinguished abelian
subcategory of $\modbar(A)$. The hypothesis $I\subseteq$ $J(A)$
implies that $\mod(A/I)$ contains all simple $A$-modules. The result
follows from Theorem \ref{AmodIembeddings}.
\end{proof}

Removing the reference to simple $A$-modules yields the following
statement.

\begin{Corollary} \label{AmodIembeddings-Cor3}
Let $A$ be a finite-dimensional selfinjective $k$-algebra such that all 
simple $A$-modules are nonprojective. Then $\modbar(A)$ has  a
semisimple distinguished abelian subcategory $\CD$ such that $\ell(\CD)$
is finite and such that for any distinguished abelian subcategory $\CD'$ 
of $\modbar(A)$ containing $\CD$ we have $\ell(\CD')=$ $\ell(\CD)$.
\end{Corollary}

\begin{proof}
Let $\CD$ be the full subcategory of $\modbar(A)$ consisting of all
semisimple modules in $\mod(A)$. The result follows from Theorem 
\ref{AmodIembeddings}.
\end{proof}

The distinguished abelian subcategories of $\modbar(A)$ of the form 
$\mod(A/I)$ in Theorem \ref{AmodIembedding1} have the property that
the simple objects in $\mod(A/I)$ remain simple in $\mod(A)$. 
The next result explores the question under what circumstances a
distinguished abelian subcategory of $\modbar(A)$ whose simple
objects correspond to simple $A$-modules is of the form $\mod(A/I)$
for some ideal $I$ in $A$.

\begin{Theorem} \label{modAmodDsimple}
Let $A$ be a finite-dimensional selfinjective $k$-algebra, and let
$D$ be a finite-dimensional $k$-algebra.
 Let $Y$ be a finitely
generated  $A$-$D$-bimodule such that $Y$ is projective as a 
right $D$-module and such that $Y$ has no nonzero projective
direct summand as a left $A$-module. 
Suppose that for any simple $D$-module $T$ the
$A$-module $Y\tenD T$ is simple and that the functor $Y\tenD -$ induces
a full embedding of $\mod(D)$ as a distinguished abelian subcategory
of $\modbar(A)$. Let $I$ be the annihilator in $A$ of $Y$ as a left
$A$-module. Then $r(I)\subseteq$ $I$, and $Y\tenD -$ induces
an equivalence $\mod(D)\cong \mod(A/I)$. 
\end{Theorem}

For the proof, we will need the following elementary observation,
which  is a sufficient criterion for an epimorphism
in the category of $k$-algebras to be an isomorphism. 

\begin{Lemma} \label{full-embed-Lemma2}
Let $D$ be a finite-dimensional $k$-algebra and $A$ a unital subalgebra
of $D$. Suppose that the restriction functor $\Res^D_A : \mod(D)\to$ 
$\mod(A)$ is a full embedding which sends every simple $D$-module 
to a simple $A$-module. Then $A=D$.
\end{Lemma}

\begin{proof}
We first show that $J(A)=$ $A\cap J(D)$. Since simple $D$-modules
restrict to simple $A$-modules,  it follows that $J(A)$ annihilates
every simple $D$-module, and hence $J(A)\subseteq$ $A\cap J(D)$.
Since $A\cap J(D)$ is a nilpotent ideal in $A$, we have the other
inclusion as well, whence the equality $J(A)=$ $A\cap J(D)$. Thus the
inclusion $A\subseteq$ $D$ induces an injective algebra homomorphism
$A/J(A)\to$ $D/J(D)$. Since $A/J(A)$ is semisimple, every $A/J(A)$-module
is injective, and hence $A/J(A)$ is isomorphic to a direct summand of
$D/J(A)$ as a right $A/J(A)$-module. Thus, for any simple $A$-module $S$,
the $D/J(D)$-module $D/J(D)\ten_{A/J(A)} S$ is nonzero. Regarded as a
left $D$-module, this is a quotient of $D\tenA S$. In particular $D\tenA S$ is
nonzero. Let $T$ be a simple quotient of $D\tenA S$ and let $D\tenA S\to$ $T$
be a nonzero $D$-homomorphism. The standard adjunction yields a nonzero
$A$-module homomorphism $S\to$ $\Res^D_A(T)$. Since $S$ and 
$\Res^D_A(T)$ are both simple $A$-modules, it follows that $S\cong$
$\Res^D_A(T)$. This shows that $\Res^D_A$ induces a bijection between 
the isomorphism classes of simple $D$-modules and simple $A$-modules.
Since $\Res^D_A$ is a full embedding, we also have $\End_D(T)=$ $\End_A(T)$.
Thus the simple modules for $D$ and $A$ which correspond to each other
through the bijection induced by $\Res^D_A$ have the same dimensions
and isomorphic endomorphism rings. The Artin-Wedderburn Theorem 
implies that $A/J(A)\cong$ $D/J(D)$, and hence $D=$ $A + J(D)$.

We show next that every maximal $A$-submodule of $D$ is
in fact a maximal $D$-submodule. Indeed, let $M$ be a maximal
$A$-submodule of $D$. Then $S=$ $D/M$ is a simple $A$-module.
By the previous argument, there is a simple $D$-module $T$ and an
$A$-module isomorphism $S\cong$ $\Res^D_A(T)$. The composition
of $A$-homomorphisms $D\to D/M=S\cong\Res^D_A(T)$ belongs to
$\Hom_A(D,T)=$ $\Hom_D(D,T)$, hence this composition is a 
$D$-homomorphism. The kernel of this $D$-homomorphism is $M$, and
hence $M$ is a maximal $D$-submodule of $D$. Conversely, any
maximal $D$-submodule $N$ of $D$ is a maximal $A$-submodule
since $\Res^D_A(D/N)$ remains simple. Taking the intersection
of all maximal submodules of $D$ as an $A$-module yields $J(A)D=J(D)$.

It follows that $D=$ $A + J(D)=$ $A + J(A)D$. Nakayama's Lemma,
applied to the $A$-module $D$, implies that $A=D$. 
\end{proof}

The converse of Lemma \ref{full-embed-Lemma2} holds trivially. One cannot
drop in this Lemma the hypothesis that $\Res^D_A$ sends simple modules 
to simple modules. Consider the subalgebra $A$ of upper triangular matrices 
in $D=M_2(k)$. The restriction from $D$ to $A$ of the unique (up to 
isomorphism) simple $D$-module is the unique (up to isomorphism) 
projective indecomposable $A$-module of dimension $2$. One verifies
easily that $\Res^D_A$ is a full embedding. (The inclusion $A\to D$ is 
thus an epimorphism in the category of rings; see 
Stenstr\"om  \cite[Chapter XI, Proposition 1.2]{Sten}.) 

\begin{Lemma} \label{embedLemma2}
Let $A$ be a finite-dimensional selfinjective algebra over a field $k$,
let $D$ be a finite-dimensional $k$-algebra, and let $Y$ be a finitely
generated $A$-$D$-bimodule. Suppose that the functor $Y\tenD-$ 
induces a full embedding $\mod(D)\to$ $\modbar(A)$. 
Then, for any finitely generated $D$-module $V$, the 
map $V\to$ $\Hombar_A(Y,Y\ten_D V)$ induced by the adjunction 
unit is an isomorphism of $D$-modules. 
In particular, the functor $\Hombar_A(Y,-) : \modbar(A)\to$ $\mod(D)$ is
a left inverse of the embedding $\mod(D)\to$ $\modbar(A)$ induced
by $Y\tenD-$. 
\end{Lemma}

\begin{proof}
By the assumptions, for any two finitely generated $D$-modules $U$, 
$V$ the map $\Hom_D(U,V)\to$ $\Hombar_A(Y\tenD U, Y\tenD V)$ 
induced by the functor $Y\tenD-$ is an isomorphism. Specialising this 
isomorphism to $U=D$ and combining it with the canonical isomorphism 
$V\cong$ $\Hom_D(D,V)$  yields the result.
\end{proof}

The converse in Lemma \ref{embedLemma2} need not hold; the issue
is that the functor $\Hombar_A(Y,-)$ need not be right adjoint to the
functor induced by $Y\tenD-$. The Tensor-Hom adjunction 
induces a natural transformation between the induced bifunctors at
the level of the stable category $\modbar(A)$ (cf. Lemma 
\ref{stable-adj-Lemma} and Remark \ref{stable-adj-Remark}), but
this need not be an isomorphism.

\begin{Lemma} \label{modulefullembed}
Let $A$, $D$ be finite-dimensional $k$-algebras and let $\Phi :
\mod(D)\to \mod(A)$ be a full exact embedding sending simple
$D$-modules to simple $A$-modules. Then $\Phi$ is isomorphic to
a functor of the form $Y\tenD - : \mod(D)\to$ $\mod(A)$, where
$Y$ is an $A$-$D$-bimodule which is a progenerator as a right
$D$-module. Moreover, if  $I$ is the annihilator in $A$ of $Y$ as a left
$A$-module, then  $\Phi$ factors through an equivalence
$\Psi : \mod(D)\cong$ $\mod(A/I)$ and the inclusion functor
$\mod(A/I)\to$ $\mod(A)$.
\end{Lemma}

\begin{proof}
The first statement is a special case of the Eilenberg-Watts Theorem:
since $\Phi$ is a full exact embedding, it is induced by tensoring with an
$A$-$D$-bimodule $Y$ which is flat as a right $D$-module. Since this is 
a functor between finite-dimensional module categories, preserving 
simple modules, it follows that $Y$ is a progenerator as a right 
$D$-module. Thus $D$ is Morita equivalent to $D'=$ 
$\End_{D^\op}(Y)$, via the functor from $\mod(D)$ to $\mod(D')$ 
 induced by $Y\tenD-$, with $Y$ here regarded as a $D'$-$D$-bimodule. 
The action of $A$ on $Y$ induces an algebra homomorphism $A\to D'$. 
Let $I$ be the annihilator 
of $Y$ in $A$. Then the algebra homomorphism $A/I\to$ $D'$ induced 
by the action of $A$ on $Y$ is injective. The functor $\Phi$ is the 
composition of the Morita equivalence $Y\tenD- : \mod(D)\to\mod(D')$ 
followed by the restriction functor along the injective algebra 
homomorphism $A/I\to D'$. By the assumptions, $\Phi$ preserves simple 
modules. Since $Y\tenD-$ induces an equivalence $\mod(D)\cong$ $\mod(D')$
it induces in particular a bijection between isomorphism classes of simple
$D$-modules and simple $D'$-modules.  It follows that simple 
$D'$-modules restrict to simple $A/I$-modules. 
Lemma \ref{full-embed-Lemma2} implies that $A/I\cong$ $D'$. Thus 
$\Phi$ factors through an equivalence $\mod(D)\to$ $\mod(A/I)$ as stated. 
\end{proof}

\begin{proof}[{Proof of Theorem \ref{modAmodDsimple}}]
Since $Y\tenD-$ induces a full embedding $\mod(D)\to$ $\modbar(A)$, it 
follows from Lemma \ref{embedLemma2}  that for any finitely generated 
$D$-modules $V$, we have an isomorphism 
$$V\cong \Hom_D(D,V) \cong \Hombar_A(Y,Y\tenD V)$$
sending $v\in$ $V$ to the image of the map $y\mapsto (y\ten v)$,
where $y\in$ $Y$. Arguing by induction over $\dim_k(V)$, we will show
that $\Hom_A^\pr(Y,Y\tenD V)=\{0\}$. If $V$ is simple, then by the
assumptions, $Y\tenD V$ is simple. Since $Y\tenD V$ is a quotient of $Y$
and since $Y$ has no nonzero projective summand as an $A$-module,
it follows that the simple $A$-module $Y\tenD V$ is nonprojective and
hence that $\Hom_A^\pr(Y,Y\tenD V)=\{0\}$. Let
$$\xymatrix{0\ar[r] & U \ar[r] & V \ar[r] & W\ar[r] & 0}$$
be a short exact sequence of nonzero $D$-modules. This sequence is 
isomorphic to the short exact sequence 
$$\xymatrix{0\ar[r] & \Hom_D(D,U) \ar[r] & \Hom_D(D,V) \ar[r] & 
\Hom_D(D,W) \ar[r] & 0}$$
Since $Y\tenD -$ induces a full embedding $mod(D)\to$ $\modbar(A)$, this
yields an exact sequence
$$\xymatrix{0\ar[r] & \Hombar_A(Y,Y\tenD U) \ar[r] & 
\Hombar_A(Y,Y\tenD V) \ar[r] &  \Hombar_A(Y,Y\tenD W) \ar[r] & 0}$$
Since $Y\tenD-$ is eaxct, the first exact sequence yields an exact sequence
$$\xymatrix{0\ar[r] & Y\tenD U \ar[r] & Y\tenD V \ar[r] & Y\tenD 
W\ar[r] & 0}$$
and applying the left exact functor $\Hom_A(Y,-)$ yields an exact sequence
$$\xymatrix{0\ar[r] & \Hom_A(Y,Y\tenD U) \ar[r] & \Hom_A(Y,Y\tenD V) 
\ar[r] & \Hom_A(Y,Y\tenD W)}$$
Thus we have a commutative exact diagram of the form
$$\xymatrix{0\ar[r] & \Hom_A(Y,Y\tenD U) \ar[r] \ar[d] & 
\Hom_A(Y,Y\tenD V)  \ar[r] \ar[d] & \Hom_A(Y,Y\tenD W) \ar[d] & \\
0\ar[r] & \Hombar_A(Y,Y\tenD U) \ar[r] & 
\Hombar_A(Y,Y\tenD V) \ar[r] &  \Hombar_A(Y,Y\tenD W) \ar[r] & 0
 }$$
where the vertical maps are the canonical surjections. Arguing by
induction, the left and right vertical maps are isomorphisms. Thus
the top right horizontal map is surjective, and comparing dimensions
implies that the middle vertical map is an isomorphism as well.
This shows that $\Hom_A^\pr(Y,Y\tenD V)=$ $\{0\}$ for all 
finitely generated $D$-modules $V$. Applied to $V=D$ this implies
that $\End_A^\pr(Y)=\{0\}$. By the first paragraph, this also implies
that the canonical map $V\to$ $\Hom_A(Y,Y\tenD V)$ is an 
isomorphism for all $V$. By Theorem \ref{full-embeddings},
the functor $Y\tenD - $ induces a full embedding $\mod(D)\to$ $\mod(A)$,
and by the assumptions, this embedding sends simple $D$-modules to
simple $A$-modules. Since $I$ is the annihilator in $A$ of $Y$, it follows
from Lemma \ref{modulefullembed}, that the full embedding
$Y\tenD - : \mod(D)\to$ $\mod(A)$ factors through an equivalence
$\mod(D)\cong$ $\mod(A/I)$. By the assumptions, the functor
$Y\tenD -$ induces a full embedding $\mod(D)\to$ $\modbar(A)$. Thus
the inclusion $\mod(A/I)\to$ $\mod(A)$ induces a full embedding
$\mod(A/I)\to$ $\modbar(A)$ as distinguished abelian subcategory.
The inclusion $r(I)\subseteq$ $I$ follows from  Theorem \ref{AmodIembedding1},
whence the result.
\end{proof}

\begin{Example} \label{AJA-example}
Let $A$ be a finite-dimensional selfinjective $k$-algebra. Suppose that 
the simple $A$-modules are non-projective. Then $J(A)$ contains its 
right annihilator $\soc(A)$ in $A$. Thus the subcategory of all 
semisimple $A$-modules, which is equivalent to
$\mod(A/J(A))$, is a distinguished abelian subcategory of $\modbar(A)$.
We have $\ell(A)=$ $\ell(A/J(A))$, so for trivial reasons, $\modbar(A)$
has distinguished abelian subcategories $\CD$ whose number of 
isomorphism classes $\ell(\CD)$ of simple objects in $\CD$ is equal to
the number $\ell(A)$ of isomorphism classes of simple $A$-modules.
\end{Example}

\begin{Example} \label{AJA2-example}
Let $A$ be a finite-dimensional selfinjective $k$-algebra.
Suppose that $\soc^2(A)\subseteq$ $J(A)^2$. Since $\soc^2(A)$ is the
right annihilator of $J(A)^2$, it follows from Theorem 
\ref{AmodIembedding1} that the composition of canonical functors 
$\mod(A/J(A)^2)\to \mod(A)\to \modbar(A)$ is an embedding of 
$\mod(A/J(A)^2)$ as a distinguished abelian subcategory in $\modbar(A)$.
If $A$ is indecomposable as an algebra, then so is $A/J(A)^2$, and both
have the same quiver. Therefore, in this situation, $\modbar(A)$ has a
{\it connected} distinguished abelian subcategory $\CD=$ 
$\mod(A/J(A)^2)$ satisfying $\ell(\CD)=$ $\ell(A)$ and $\Ext^1_\CD(S,T)
\cong$ $\Ext^1_\CC(S,T)$, for any two simple objects $S$, $T$ in $\CD$. 
\end{Example}

\begin{Remark} \label{soc2A-Remark}
The property $\soc^2(A)\subseteq$ $J(A)^2$ in the previous Example
is not invariant under stable equivalences, in fact, not even 
under derived equivalences. For instance, the Brauer tree algebra of a 
star with four edges has this property, but the Brauer tree algebra of a 
line with four edges (and no exceptional vertex) does not. 
\end{Remark}

\section{Distinguished abelian subcategories for finite group algebras}
\label{groupSection}

We describe special cases of the situation arising in Theorem 
\ref{AmodIembedding1} involving finite group algebras. We use without
further comments standard properties of finite $p$-group algebras in 
prime characteristic $p$; see e. g. \cite[Section 1.11]{LiBookI}.

\begin{Theorem} \label{AkZexample}
Let $k$ be a field of prime characteristic $p$ and $A$ a 
finite-dimensional selfinjective $k$-algebra. Suppose that $Z(A)^\times$ 
has a nontrivial finite $p$-subgroup $Z$ such that $A$ is projective as 
a $kZ$-module. Set $I=$ $I(kZ)\cdot A$, where $I(kZ)$ is the 
augmentation ideal of $kZ$. Then $I$ contains its right annihilator in 
$A$. In particular, restriction along the canonical surjection $A\to$ 
$A/I$ induces a full embedding $\mod(A/I)\to$ $\modbar(A)$ of 
$\mod(A/I)$ as a distinguished abelian  subcategory in $\modbar(A)$.
\end{Theorem}

\begin{proof}
The right annihilator of $I(kZ)$ in $kZ$ is the 
$1$-dimensional ideal $\soc(kZ)=$ $(\sum_{z\in Z} z)\cdot kZ$, and we 
have $\soc(kZ)\subseteq$ $I(kZ)$. Since $A$ is a free left or right 
$kZ$-module, an easy argument shows that the right annihilator of 
$I(kZ)\cdot A=$ $A\cdot I(kZ)$ is therefore $\soc(kZ)\cdot A=$ 
$A\cdot \soc(kZ)$, which is contained in $I(kZ)\cdot A$. Thus the 
statement is the special case of Theorem \ref{AmodIembedding1} with 
$I=$ $I(kZ)\cdot A$.
\end{proof}

Let $G$ be a finite group. 
The module category $\mod(kG)$ of the finite group algebra $kG$ over
a field $k$ is a symmetric monoidal category with respect to the tensor 
product $-\ten_k-$ of $kG$-modules over $k$. It is well-known that if 
$U$, $V$ are finitely generated $kG$-modules with at least one of $U$, 
$V$ projective, then $U\tenk V$ is projective as well. Therefore the 
tensor product over $k$ induces a commutative monoidal structure on the 
triangulated category $\modbar(kG)$. If $N$  is a normal subgroup of 
$G$, then the canonical surjection $G\to$ $G/N$ induces an embedding of 
symmetric monoidal categories $\mod(kG/N)\to$ $\mod(kG)$. The
following result implies the first statement in Theorem \ref{kGmodN}. 

\begin{Theorem} \label{kGexample}
Let $k$ be a field of prime characteristic $p$ and $G$ a finite group. 
Let $N$ be a normal subgroup of $G$. Restriction along 
the canonical surjection $G\to$ $G/N$ induces a full embedding 
$\mod(kG/N)\to$ $\modbar(kG)$ of $\mod(kG/N)$ as a symmetric monoidal 
distinguished abelian subcategory in $\modbar(kG)$ if and only
if $p$ divides the order of $N$.
\end{Theorem}

\begin{proof}
The fact that the functor $\mod(kG/N)\to$ $\modbar(kG)$ is a functor
of symmetric monoidal categories is obvious (see the remarks preceding
the Theorem). We need to show that this induces an embedding as a 
distinguished abelian subcategory in $\modbar(kG)$ if and only if
$|N|$ is divisible by $p$. 
The kernel of the canonical algebra homomorphism $kG\to$ $kG/N$ is equal 
to $I=$ $kG\cdot I(kN)$, where $I(kN)$ is the augmentation ideal of 
$kN$. Arguing as in the previous proof, the right annihilator of $I(kN)$ 
in $kN$ is the $1$-dimensional ideal $(\sum_{y\in N}\ y) kN$. This is 
contained in $I(kN)$ if and only if $p$ divides $|N|$. Indeed, if $p$ 
divides $|N|$, then $\sum_{y\in N} y = $ $\sum_{y\in N} (y-1)\in$ 
$I(kN)$. If $p$ does not divide $|N|$, then $(\sum_{y\in N}\ y) kN$ is 
a complement of $I(kN)$ in $kN$. Since $kG$ is free as a right 
$kN$-module of rank $|G:N|$, it follows that the right annihilator of 
$I$ is equal to $kG\cdot(\sum_{y\in N} y)$. Therefore, if $p$ divides 
$|N|$, then the right anihilator of $I$ is contained in $I$ by the 
previous argument. The result follows in that case from Theorem 
\ref{AmodIembedding1}. 
If $|N|$ is prime to $p$, then $kG/N$ is a projective $kG$-module, so 
$\Endbar_{kG}(kG/N)$ vanishes, and in particular, the canonical
functor $\mod(kG/N)\to$ $\modbar(kG)$ is not an embedding.

Alternatively, one can also show this using a special case of Higman's 
criterion. Let $U$, $V$ be $kG/N$-modules. When regarded as 
$kG$-modules, the elements of $N$ act as identity on $U$, $V$. Thus any 
$k$-linear map $\tau : U\to$ $V$ is a $kN$-homomorphism, and 
$\Tr^G_1(\tau)=$ $|N|\Tr^G_N(\tau)$. If $p$ divides $|N|$, then
this expression is zero in $k$. It follows from the special case 
\cite[Proposition 2.13.11]{LiBookI} of Higman's criterion that 
$\Hom_{kG}^\pr(U,V)=$ $\{0\}$. 
Equivalently, we have $\Hombar_{kG}(U,V)\cong$ $\Hom_{kG}(U,V)=$
$\Hom_{kG/N}(U,V)$. This shows that if $p$ divides $|N|$, then 
$\mod(kG/N)$ can indeed be identified canonically with a full 
subcategory of $\modbar(kG)$. 
Note that $kG/N$ is a projective $kG/N$-module. Thus every
$kG/N$-endomorphism of $kG/N$ is equal to $\Tr^{G/N}_1(\sigma)$ for
some linear endomorphism $\sigma$ of $kG/N$. Equivalently, every
$kG$-endomorphism of $kG/N$ is of the form $\Tr^G_N(\sigma)$ for some
$kN$-endomorphism $\sigma$ of $kG/N$. If $|N|$ is coprime to $p$, then 
$\sigma=$ $\frac{1}{|N|}\Tr^N_1(\sigma)$, hence $\tau=$ 
$\Tr^G_1(\frac{1}{|N|}\sigma)$, which shows that $\tau$ factors through
a projective $kG$-module. Equivalently, the canonical map
$\End_{kG/N}(kG/N)\to$ $\Endbar_{kG}(kG/N)$ is zero. This shows that 
if $|N|$ is coprime to $p$, then the canonical functor $\mod(kG/N)\to$ 
$\modbar(kG)$ is not an embedding.
\end{proof}

\begin{Remark} \label{normalpRemark}
Let $k$ be a field of prime characteristic $p$.
\begin{enumerate}
\item
Let $G$ be a finite group having a nontrivial normal $p$-subgroup $Q$. 
It is well-known that the kernel $I$ of the canonical algebra 
homomorpism $kG\to$ $kG/Q$ is contained in the radical $J(kG)$ and hence 
that $\ell(kG)=$ $\ell(kG/Q)$. Thus Theorem \ref{kGexample} illustrates
Theorem \ref{AmodIembeddings}, constructing explicitly 
the distinguished abelian  subcategory $\mod(kG/Q)$ of $\modbar(kG)$
whose number of isomorphism classes of simple  objects is equal to that 
of $\mod(kG)$.

\item
Theorem \ref{kGexample} implies that if $P$ is a nontrivial finite 
$p$-group, then any cyclic subgroup of $Z(P)$ yields a distinguished 
abelian subcategory of $\modbar(kP)$. But then so does any shifted 
cyclic subgroup of $Z(P)$, suggesting that distinguished abelian 
subcategories should form varieties which are related to cohomology 
support varieties. 
\end{enumerate}
\end{Remark}

Combining Theorem \ref{modAmodDsimple}, a result of J. F. Carlson 
\cite[Theorem 1]{Car98}, and \cite[Theorem 3.4]{Listable} yields the 
following classification of those distinguished abelian  subcategories 
of the stable module category of a finite $p$-group algebra in prime 
characteristic $p$ which are equivalent to module categories of 
finite-dimensional split $k$-algebras.
If $P$ is a finite $p$-group, then a finiteley generated $kP$-module
$V$ is called {\it endotrivial} if $V\tenk V^*\cong$ $k \oplus U$ for
some projective $kP$-module. If $V$ is endotrivial, then $V\tenk -$ and 
$V^*\tenk-$ induce inverse equivalences on $\modbar(kP)$. In particular,
$V\tenk-$ sends in that case any distinguished abelian subcategory
$\CD$ of $\modbar(kP)$ to a distinguished abelian subcategory,
denoted $V\tenk \CD$, of $\modbar(kP)$.

\begin{Theorem} \label{Pgroupdist}
Let $p$ be a prime, $P$ a nontrivial finite $p$-group and $k$ a field 
of characteristic $p$. Let $D$ be a finite-dimensional split basic
$k$-algebra such that there is an embedding $\Phi : \mod(D)\to$ 
$\modbar(kP)$ as distinguished abelian subcategory of $\modbar(kP)$.

\begin{enumerate}
\item[{\rm (i)}] 
We have $\ell(D)=1$; that is, $D$ is split local.

\item[{\rm (ii)}] 
Let $V$ is an indecomposable $kP$-module corresponding to a simple
$D$-module under the functor $\Phi$. Then $V$ is an endotrivial
$kP$-module.

\item[{\rm (iii)}]
If $\Phi$ is induced by a functor $Y\ten_D-$ for some finitely
generated $kP$-$D$-bimodule $Y$ which is projective as a right
$D$-module, then there is an ideal $I$ of $kP$ containing its right
annihilator in $kP$ such that $D\cong$ $(kP)/I$ and such that
$\Phi$ induces an equivalence between $\mod(D)$ and the distinguished
abelian subcategory $V\tenk \mod((kP)/I)$. 
\end{enumerate}
\end{Theorem}

The first statement of Theorem \ref{Pgroupdist} holds slightly more
generally, based on the following observation which is a consequence
of the proof of \cite[Theorem 3.4]{Listable}. Slightly extending 
standard terminology from finite-dimensional algebras, if $\CC$ is
a $k$-linear triangulated category, then a distinguished abelian 
subcategtory $\CD$ of $\CC$ is called {\it split} if for every object 
$X$ in $\CD$ which is simple as an object of $\CD$ we have 
$\End_\CD(X)\cong$ $k$. 

\begin{Lemma} \label{kPdistone}
Let $p$ be a prime, $P$ a nontrivial finite $p$-group, and suppose that
$\chr(k)=$ $p$. Let $\CD$ be a split distinguished abelian subcategory
of $\modbar(kP)$ such that $\CD$ has a simple object. Then $\ell(\CD)=1$.
\end{Lemma}

\begin{proof}
By the assumptions on $\CD$ we have $\ell(\CD)\geq$ $1$ (we include
here by convention the case where $\CD$ has infinitely many isomorphism
classes of simple objects). Arguing  by contradiction, suppose that 
$\ell(\CD)\geq$ $2$. Thus $\CD$  has two nonisomorphic simple objects 
$S$, $T$. Since $\CD$ is a full subcategory of $\modbar(kP)$, the 
objects $S$, $T$ remain nonisomorphic in $\modbar(kP)$. Since $\CD$ is 
split, we have $\Endbar_{kP}(S)\cong$ $k\cong$ $\Endbar_{kP}(T)$, and 
we have $\Hombar_{kP}(S,T)=\{0\}=$ $\Hombar_{kP}(T,S)$. It is shown in 
the proof of \cite[Theorem 3.4]{Listable} that this is not possible.
\end{proof} 

\begin{proof}[{Proof of Theorem \ref{Pgroupdist}}]
Denote by $\CD$ the distinguished abelian subcategory of $\mod(kP)$
obtained from taking the closure under isomorphisms in $\modbar(kP)$
of the image of the embedding $\Phi : \mod(D)\to$ $\modbar(kP)$.
By Lemma \ref{kPdistone}, the category $\CD$ has a unique isomorphism 
class of simple objects, whence (i). 
Let $V$ be an indecomposable $kP$-module such 
that $V$ is simple as an object in $\CD$. Then in particular 
$\Endbar_{kP}(V)\cong$ $k$. A result of J. F. Carlson 
\cite[Theorem 1]{Car98} implies that $V$ is endotrivial, which shows 
(ii). The exact functor $V\tenk -$ on $\mod(kP)$
induces an equivalence on $\modbar(kP)$, with inverse induced by the
functor $V^*\tenk -$. Thus after replacing $\CD$ by the image
of $\CD$ under the functor $V^*\tenk-$, we may (and do) assume that
the trivial $kP$-module $k$ belongs to $\CD$, and is the - up to
isomorphism unique - simple object of $\CD$. In other words, with the
notation and hypotheses of statement (iii), the $kP$-$D$-bimodule
$Y$ is projective as a right $D$-module and the functor $Y\ten_D-$
sends a simple $D$-module to the trivial $kP$-module. Thus the
hypotheses of Theorem \ref{modAmodDsimple} are satisfied, implying
statement (iii).
\end{proof}

\begin{Example} \label{nonproperExample}
Let $p=2$, let $P$ be a finite $2$-group of order at least $4$, let $Z$ be a subgroup of
order $2$ of $Z(P)$, and set $Q=P/Z$. By Theorem \ref{kGexample}, the abelian category 
$\mod(kQ)$ can be identified with a distinguished abelian subcategory of $\modbar(kP)$. 
This subcategory is not proper in the sense of \cite[Def. 1.2]{Jorg20}. 
Tensoring the obvious short exact sequence of $kZ$-modules
$$\xymatrix{0 \ar[r] & k \ar[r] & kZ \ar[r] & k \ar[r] & 0}$$
by $kP\ten_{kZ}-$ yields a short exact sequence of $kP$-modules
$$\xymatrix{0 \ar[r] & kQ \ar[r]  & kP \ar[r]  & kQ \ar[r]  & 0}\ .$$
In other words, regarding $kQ$ as a $kP$-module via the canonical surjection $P\to Q$
implies that $\Sigma(kQ)\cong kQ$. Thus we have a distinguished exact triangle in
$\modbar(kP)$ of the form
$$\xymatrix{kQ \ar[r] & 0 \ar[r] &  kQ \ar@{=}[r] & kQ }$$
The first three terms of this triangle belong to $\mod(kQ)$ but do not form a short 
exact sequence in $\mod(kQ)$.
\end{Example}  

\section{Basic properties of distinguished abelian 
subcategories} \label{basicpropertySection}

We show that a short exact sequence in a distinguished abelian
subcategory $\CD$ of a triangulated category $\CC$ determines a 
unique exact triangle in $\CC$; that is, we show that the morphism 
$h$ in Definition \ref{distDef} is unique. If not stated otherwise, the
shift functor of a triangulated category is denoted by $\Sigma$.

\begin{Proposition} \label{hunique}
Let $\CC$ be a triangulated category and let $\CD$ be a
distinguished abelian subcategory of $\CC$. Let 
$$\xymatrix{0 \ar[r] & X\ar[r]^{f} & Y \ar[r]^{g} &Z\ar[r] & 0}$$
be a short exact sequence in $\CD$. There is a unique morphism
$h : Z\to \Sigma(X)$ in $\CC$ such that 
$$\xymatrix{X\ar[r]^{f} & Y \ar[r]^{g} &Z\ar[r]^{h} & \Sigma(X)}$$
is an exact triangle in $\CC$. 
\end{Proposition}

\begin{proof}
The existence of $h$ is clear by definition; we need to show the 
uniqueness. Let $h$, $h' : Z\to$ $\Sigma(X)$ be morphisms in $\CC$ 
such that the triangles 
$$\xymatrix{X\ar[r]^{f} & Y \ar[r]^{g} &Z\ar[r]^{h} & \Sigma(X)}$$
$$\xymatrix{X\ar[r]^{f} & Y \ar[r]^{g} &Z\ar[r]^{h'} & \Sigma(X)}$$
are exact. The pair of identity morphisms $(\Id_X, \Id_Y)$ can be
completed to a morphism of triangles $(\Id_X, \Id_Y, a)$. That is, 
there is a morphism $a : Z\to$ $Z$ satisfying $a\circ g=$ $g$ and
$h'\circ a=$ $h$. Since $Z$ belongs to $\CD$ and since $\CD$ is a 
full subcategory of $\CC$, it follows that $a$ is a morphism in the 
abelian category $\CD$. Since $g$ is an epimorphism in $\CD$, this 
forces $a=$ $\Id_Z$, whence $h'=h$. 
\end{proof}

\begin{Remark} \label{closureRemark}
The definition of a distinguished abelian subcategory $\CD$ of a 
triangulated category $\CC$ does not require $\CD$ to be closed under
isomorphisms in $\CC$. One easily checks that the closure of $\CD$
under isomorphisms in $\CC$ is again a distinguished abelian
subcategory of $\CC$ which is equivalent to $\CD$ as an abelian 
category. 
\end{Remark}

\begin{Proposition} \label{detectProp}
Let $\CC$ be a triangulated category and let $\CD$ be a
distinguished abelian subcategory of $\CC$. Let
$$\xymatrix{0 \ar[r] & X \ar[r]^{f} & Y \ar[r]^{g} & Z \ar[r] & 0}$$
be  a short exact sequence in $\CD$ and let $W$ be an object in $\CC$.
If $W$ belongs to $\CD$, then  the maps
$$\xymatrix{\Hom_\CC(\Sigma(W), Y) \ar[r] & \Hom_\CC(\Sigma(W), Z) }$$
$$\xymatrix{\Hom_\CC(\Sigma(Y), W) \ar[r] & \Hom_\CC(\Sigma(X), W) }$$
induced by composition with $g$ and precomposition with $\Sigma(f)$
are surjective.
\end{Proposition}

\begin{proof}
Since $\CD$ is a distinguished abelian subcategory in $\CC$,  
there is a morphism $h : Z\to$ $\Sigma(X)$ such that the triangle
$$\xymatrix{X\ar[r]^{f} & Y \ar[r]^{g} &Z\ar[r]^{h} & \Sigma(X)}$$
in $\CC$ is exact. Applying the functor $\Hom_\CC(W,-)$ yields
a long exact sequence of  the form 
{\small{
$$\xymatrix{\cdots\ar[r] & \Hom_\CC(W,\Sigma^{-1}(Y)) \ar[r] 
&\Hom_\CC(W,\Sigma^{-1}(Z)) \ar[r] & \Hom_\CC(W,X) \ar[r] 
&\Hom_\CC(W,Y)\ar[r] & \cdots}$$
}}
The right map is induced by composing with the monomorphism $f$ in
$\CD$. Thus if $W$ belongs to $\CD$, then the right map is injective.
But then the map in the middle is zero, so the left map is surjective.
Since $\Sigma$ is an equivalence, it follows that the map
$\xymatrix{\Hom_\CC(\Sigma(W), Y) \ar[r] & \Hom_\CC(\Sigma(W), Z) }$
is surjective. Similarly, we have a long exact sequence
{\small{
$$\xymatrix{\cdots\ar[r] & \Hom_\CC(\Sigma(Y),W) \ar[r] 
&\Hom_\CC(\Sigma(X),W) \ar[r] & \Hom_\CC(Z,W) \ar[r] 
&\Hom_\CC(Y,W)\ar[r] & \cdots}$$
}}
The right map is induced by precomposing with the epimorphism $g$
in $\CD$. Thus if $W$ belongs to $\CD$, then the right map is 
injective, hence the map in the middle is zero, and therefore the
left map is surjective. This concludes the proof.
\end{proof}

Unlike hearts of $t$-structures, distinguished abelian subcategories 
need not be disjoint from their shifts - they may contain periodic 
objects. The following consequence of Proposition \ref{detectProp}
shows that if $\CD$ is a 
distinguished abelian subcategory in a triangulated category $\CC$, 
then $\CD\cap \Sigma(\CD)$ is a subcategory of the additive category 
$\proj(\CD)$ generated by the projective objects in $\CD$.

\begin{Corollary} \label{WSigmaW}
Let $\CC$ be a triangulated category and let $\CD$ be a 
distinguished abelian subcategory of $\CC$. Let $W$ be an object in 
$\CD$ such that $\Sigma(W)$ is an object in $\CD$. Then $W$ is 
injective in $\CD$ and $\Sigma(W)$ is projective in $\CD$.
\end{Corollary}

\begin{proof}
By Proposition \ref{detectProp}, if $g : Y\to Z$ is an epimorphism in $\CD$,
then every morphism $\Sigma(W)\to$ $Z$
lifts through $g$. Since $\Sigma(W)$ belongs to $\CD$, it follows that 
$\Sigma(W)$ is projective in $\CD$. Similarly, by Proposition \ref{detectProp} 
(applied with $\Sigma(W)$ instead of $W$), if $f : X\to Y$ is a 
monomorphism in $\CD$, then every morphism $X\to W$ factors through 
$f$, and hence $W$ is injective in $\CD$.
\end{proof}

\begin{Proposition} \label{monoepiProp}
Let $\CC$ be a triangulated category and let $\CD$ be a
distinguished abelian subcategory of $\CC$. Let 
$$\xymatrix{0 \ar[r] & X \ar[r]^{f} & Y \ar[r]^{g} & Z \ar[r] & 0}$$
be  a short exact sequence in $\CD$ and let $W$ be an object in $\CC$.
\begin{enumerate}
\item[{\rm (i)}]
If $W$ is a projective object in $\CD$, then the map
$$\Hom_\CC(W,\Sigma(X))\to \Hom_\CC(W,\Sigma(Y))$$
induced by composition with $\Sigma(f)$ is injective.
\item[{\rm (ii)}]
If $W$ is an injective object in $\CD$, then the map
$$\Hom_\CC(Z,\Sigma(W)) \to \Hom_\CC(Y,\Sigma(W))$$
induced by precomposition with $g$ is injective.
\end{enumerate}
\end{Proposition}

\begin{proof}
Since $\CD$ is a distinguished abelian subcategory in $\CC$,  
there is a morphism $h : Z\to$ $\Sigma(X)$ such that the triangle
$$\xymatrix{X\ar[r]^{f} & Y \ar[r]^{g} &Z\ar[r]^{h} & \Sigma(X)}$$
in $\CC$ is exact. Applying the functor $\Hom_\CC(W,-)$ yields a
long exact sequence of the form 
$$\xymatrix{\cdots\ar[r] & \Hom_\CC(W,Y)\ar[r]^{g_*} & \Hom_\CC(W,Z) \ar[r]^{h_*} 
&\Hom_\CC(W,\Sigma(X)) \ar[r]^{\Sigma(f)_*} 
& \Hom_\CC(W,\Sigma(Y)) \ar[r] &\cdots}$$
If $W$ is projective in $\CD$, then $g_*$ is surjective, hence  
$h_*$ is zero.  This implies that $\Sigma(f)_*$ is injective, proving 
(i). A dual argument, applying the functor $\Hom_\CC(-,W)$, and using 
the fact that $\Sigma$ is an equivalence, shows (ii).
\end{proof}

A category $\CD$ is called {\it split} if every morphism $f : X\to Y$
in $\CD$ is {\it split}; that is, if there exists a morphism 
$g : Y\to X$ such that $f=$ $f\circ g\circ f$.  If $\CD$ is an abelian
category, an easy verification shows that $\CD$ is split if and only if 
every monomorphism (resp. epimorphism) in $\CD$ is split. 
All epimorphisms and monomorphisms in a triangulated category are
split. 

\begin{Proposition} \label{no-adjoint}
Let $\CC$ be a triangulated category and $\CD$ a distinguished
abelian subcategory. If the inclusion functor $\CD\subseteq$ $\CC$ has
a left adjoint or a right adjoint as an additive functor, then $\CD$ is split.
\end{Proposition}

\begin{proof}
Suppose that the inclusion functor $\CD\subseteq$ $\CC$ has a left
adjoint $\Phi$. That is, for any object $U$ in $\CD$ and any object
$X$ in $\CC$ we have a natural isomorphism $\Hom_\CD(\Phi(X),U)\cong$
$\Hom_\CC(X,U)$. Thus any monomorphism $U\to U'$ in $\CD$
induces an injective map $\Hom_\CC(X,U)\to$ $\Hom_\CC(X,U')$.
This shows that the morphism $U\to U'$ is a monomorphism in $\CC$,
hence split in $\CC$. Since $\CD$ is a full subcategory of $\CC$ it follows
that the monomorphism $U\to U'$ is split in $\CD$, and hence
$\CD$ is split. A similar argument
shows that if the inclusion functor $\CD\subseteq$ $\CC$ has a right adjoint,
then every epimorphism in $\CD$ is split, whence the result.
\end{proof}

If $D$ is a finite-dimensional $k$-algebra, then $D$
is semisimple if and only if $\mod(D)$ is split. Thus Proposition
\ref{no-adjoint} has the following immediate consequence.

\begin{Corollary} \label{no-adjoint-Cor}
Let $A$ be a finite-dimensional selfinjective $k$-algebra, $D$ a 
finite-dimensional $k$-algebra, and $\Phi : \mod(D) \to$
$\modbar(A)$ a full embedding of $\mod(D)$ as a distinguished
abelian subcategory in $\modbar(A)$. If $\Phi$ has a left adjoint
or a right adjoint, then $D$ is semisimple.
\end{Corollary}

This Corollary implies in particular that even if $\Phi$ is induced by 
tensoring with a suitable $A$-$D$-bimodule, the Tensor-Hom 
adjunction does not in general yield a right adjoint to $\Phi$; see
Lemma \ref{stable-adj-Lemma} and Remark \ref{stable-adj-Remark}
for some more comments.

\begin{Remark} \label{completionRemark}
By a result of Balmer and Schlichting \cite[Theorem 1.5]{BaSch}, the 
idempotent completion $\hat\CC$ of a triangulated category $\CC$ is 
triangulated in such a way that the canonical embedding $\CC\to$ 
$\hat\CC$ is an exact functor. Since this embedding is full, it follows 
that a distinguished abelian subcategory $\CD$ of $\CC$ remains a 
distinguished abelian subcategory of $\hat\CC$. Since any abelian
category is idempotent split, it follows that the indecomposable 
objects in $\CD$ remain indecomposable in $\hat\CC$.
\end{Remark}

\begin{Example} \label{progenExample}
Let $k$ be a field of characteristic $2$ and let $P$ be a finite
$2$-group of order at least $4$. Let $Z$ be a central subgroup of order 
$2$ of $P$. Set $Y=$ $kP/Z$ as left $kP$-module. Then $\End_{kP}^\pr(Y)=$
$\{0\}$ and $\End_{kP}(Y)\cong$ $(kP/Z)^\op$. Clearly $Y$ is a 
progenerator of the distinguished abelian subcategory $\mod(kP/Z)$ of 
$\modbar(kP)$, obtained from the restriction functor $\Phi$ given by the 
canonical surjection $kP\to$ $kP/Z$. The functor $\Phi$ is trivially 
isomorphic to $Y\ten_{kP/Z}-$.  As mentioned in Example \ref{nonproperExample},
we have $\Sigma(Y)\cong$ $Y$, where 
$\Sigma$ is the shift functor in $\modbar(kP)$. In other words, as a 
$kP$-module, $Y$ has period $1$. Therefore, $Y$ is also a progenerator 
of the distinguished abelian subcategory $\Sigma(\mod(kP/Z))$. The 
subcategories $\mod(kP/Z)$ and $\Sigma(\mod(kP/Z))$ are different; in 
fact, their intersection is $\add(Y)$ because of Proposition 
\ref{WSigmaW}. In particular, the embedding $\Sigma\circ\Phi : 
\mod(kP/Z)\to\modbar(kP)$ is not induced by the functor $Y\ten_{kP/Z}-$.
It is, though, still induced by tensoring with a bimodule, namely
the $kP$-$kP/Z$-bimodule $\Sigma_{P\times P/Z}(Y)$. This is because we
have composed the embedding $\Phi$ with the self-equivalence $\Sigma$ on
$\modbar(kP)$, which is a stable equivalence of Morita type, hence
induced by tensoring with a suitable bimodule. One should expect that
composing $\Phi$ with a stable equivalence on $\modbar(kP)$ which is not 
of Morita type would yield embeddings $\mod(kP/Z)\to$ $\modbar(kP)$ as
distinguished abelian subcategories which are not induced by tensoring 
with any bimodule.
\end{Example}

\section{Exact sequences in distinguished abelian subcategories}
\label{exactSection}

Any exact triangle $X\to Y\to Z\to \Sigma(X)$ in a triangulated category 
$\CC$ such that $X$, $Y$, $Z$ belong to the heart $\CA$ of a 
$t$-structure is in fact induced by a short exact sequence 
$0\to X\to Y\to Z\to 0$ in $\CA$. In an arbitrary distinguished abelian 
subcategory, this need not be the case. The case where a nonzero object 
$W$ and its shift $\Sigma(W)$ belong to a distinguished abelian 
subcategory $\CD$ of $\CC$ yields an exact triangle
$$\xymatrix{W \ar[r] & 0 \ar[r] & \Sigma(W) \ar@{=}[r] & \Sigma(W)}$$
which is not induced by an exact sequence in $\CD$. As noted in 
Corollary \ref{WSigmaW}, in that situation $W$ is injective in $\CD$
and $\Sigma(W)$ is projective in $\CD$. This situation arises in Example 
\ref{nonproperExample}.

The following result shows that these are essentially the only exact
triangles with three terms in $\CD$ which can arise besides
those induced by short exact sequences in $\CD$. As before, we 
denote the shift functor in a triangulated category by $\Sigma$.

\begin{Proposition} \label{DexactC}
Let $\CD$ be a distinguished abelian subcategory in a triangulated
category $\CC$, and let
$$\xymatrix{X \ar[r]^{f} & Y \ar[r]^{g} & Z \ar[r]^{h} & \Sigma(X)}$$
be an exact triangle in $\CC$ such that $X$, $Y$, $Z$ belong to $\CD$. 
Then this triangle is isomorphic to a direct sum of two exact triangles 
of the form
$$\xymatrix{X' \ar[r]^{f'} & Y \ar[r]^{g'} & Z' \ar[r]^{h'} 
& \Sigma(X')}$$
$$\xymatrix{W \ar[r] & 0 \ar[r] & \Sigma(W) \ar@{=}[r] & \Sigma(W)}$$
where $X'$, $Z'$, $W$, $\Sigma(W)$ are in $\CD$, and where the sequence
$$\xymatrix{0\ar[r] & X' \ar[r]^{f'} & Y \ar[r]^{g'} & Z' \ar[r] &0}$$
is exact in $\CD$. Moreover, $W$ is injective in $\CD$ and $\Sigma(W)$ 
is projective in $\CD$.
\end{Proposition}

\begin{proof}
Since $\CD$ is abelian, the morphism $g$ has a kernel $f' : X'\to Y$ 
in $\CD$. Then in particular $g\circ f'=0$, and hence 
there is a morphism $v : X'\to X$ such that $f\circ v=f'$. Since $f'$ 
is the kernel of $g$ and since $g\circ f=0$, there is a unique 
morphism $w : X \to X'$ such that $f = f'\circ w$. Thus 
$$f'\circ w \circ v = f\circ v = f'$$
and since $f'$ is a monomorphism, this forces $w\circ v=$ $\Id_{X'}$.
Denote by $g' : Y\to Z'$ a cokernel of $f'$ in $\CD$, so that we get a 
short exact sequence
$$\xymatrix{0\ar[r] & X' \ar[r]^{f'} & Y \ar[r]^{g'} & Z' \ar[r] &0}$$
in $\CD$. Since $\CD$ is distinguished, this can be completed to an 
exact triangle in $\CC$ with a morphism $h' : Z'\to$ $\Sigma(X')$. 
The morphisms $v$ and $w$ yield morphisms of triangles
$$\xymatrix{X' \ar[r]^{f'} \ar[d]_{v} & Y\ar[r]^{g'} \ar@{=}[d] 
& Z'\ar[r]^{h'} \ar[d]^{a} & \Sigma(X') \ar[d]^{\Sigma(v)} \\
X \ar[r]^{f} \ar[d]_{w} & Y\ar[r]^{g} \ar@{=}[d] 
& Z\ar[r]^{h} \ar[d]^{b} & \Sigma(X) \ar[d]^{\Sigma(w)} \\
X' \ar[r]_{f'} & Y\ar[r]_{g'} & Z'\ar[r]^{h'} & \Sigma(X')}$$
Since $w\circ v=\Id_{X'}$, it follows that $b\circ a$ is an
automorphism of $Z'$. Therefore $(\Id_{X'}, \Id_Y, b\circ a)$ is an
automorphism of the third triangle, and hence so is its inverse. 
After replacing $b$ by $(b\circ a)^{-1}\circ b$, we therefore may 
choose $b$ in such a way that $b\circ a=$ $\Id_{Z'}$. 
It follows that the first triangle is a direct 
summand of the second, and that it has a complement isomorphic to
$$\xymatrix{W \ar[r] & 0 \ar[r] & \Sigma(W) \ar@{=}[r] & \Sigma(W)}$$
where $W$ is the complement of $X'$ in $X$ determined by $\ker(w)$.
The last statement on $W$ (resp. $\Sigma(W)$) being injective (resp. 
projective) in $\CD$ follows from Corollary \ref{WSigmaW}.
\end{proof}

\begin{Corollary} \label{triangsubcatCor}
Let $\CC$ be a triangulated category, $\CD$ a distinguished abelian
subcategory of $\CC$, and $\CT$ a thick subcategory of
$\CC$. Suppose that $\CT\subseteq$ $\CD$. Then all objects in $\CT$ 
are projective and injective in $\CD$, and $\CT$ is split. 
\end{Corollary}

\begin{proof}
Since $\CT$ is closed under powers of $\Sigma$, it follows from 
Corollary \ref{WSigmaW} that all objects in $\CT$ are projective and injective
in $\CD$. Let $f : X\to$ $Y$ be a morphism in $\CT$. 
Complete $f$ to an exact triangle
$$\xymatrix{X\ar[r]^{f} & Y \ar[r]^{g} &Z\ar[r]^{h} & \Sigma(X)}$$
in $\CT$ (or equivalently, in $\CC$).  Since $\CT$ is contained in $\CD$, it
follows that (possibly after replacing $Z$ by an isomorphic object) 
$Z$ belongs to $\CD$, and hence the morphism $g$ belongs
to $\CD$. This triangle is the direct sum of two triangles as in 
Proposition \ref{DexactC}. All terms in these two triangles are
in $\CT$, hence projective and injective in $\CD$. The first of the two
triangles is induced by a short exact sequence in $\CD$, and therefore 
split. The second of the two triangles is trivially split.
The result follows.
\end{proof}

\begin{Remark} \label{triang-sym-Remark}
With the notation of Corollary \ref{triangsubcatCor}, suppose that 
$\CD$ is equivalent to $\mod(D)$ for some finite-dimensional symmetric
$k$-algebra $D$ and that $\CT$ is a thick subcategory of $\CC$ which is
contained in $\CD$. Then $\CT$ is generated, as an additive category, by 
indecomposable projective (or equivalently, injective) objects in $\CD$. By 
the assumption on $\CD$, each projective indecomposable object $U$ in 
$\CD$ has an endomorphism with image the socle of $U$. This 
endomorphism is split, hence an isomorphism, and thus $U$ is simple. 
Therefore, in this situation, $\CT$  consists of projective semisimple 
objects in $\CD$ which are permuted by $\Sigma$.
\end{Remark}

\begin{Proposition} \label{abelianinclusion}
Let $\CC$ be a triangulated category and let $\CD$, $\CD'$ be
distinguished abelian subcategories of $\CC$ such that $\CD\subseteq$
$\CD'$. Then $\CD$ is an abelian subcategory of $\CD'$; that is, the
inclusion functor $\CD\subseteq$ $\CD'$ is exact.
\end{Proposition}

\begin{proof}
Let 
$$\xymatrix{0\ar[r] & X \ar[r]^{f} & Y \ar[r]^{g} & Z \ar[r] &0}$$
be a short exact sequence in $\CD$. We need to show that this sequence
remains exact in $\CD'$. Since $\CD$ is a distinguished abelian
subcategory of $\CC$, it follows that there is a morphism
$h : Z\to\Sigma(X)$ in $\CD$ such that the triangle 
$$\xymatrix{X \ar[r]^{f} & Y \ar[r]^{g} & Z \ar[r]^{h} & \Sigma(X)}$$
in $\CC$ is exact. The objects $X$, $Y$, $Z$ belong to $\CD$, hence to
$\CD'$. By Proposition \ref{DexactC}, this triangle is a
direct sum of exact triangle of the form 
$$\xymatrix{X' \ar[r]^{f'} & Y \ar[r]^{g'} & Z' \ar[r]^{h'} 
& \Sigma(X')}$$
$$\xymatrix{W \ar[r] & 0 \ar[r] & \Sigma(W) \ar@{=}[r] & \Sigma(W)}$$
where $X'$, $Z'$, $W$, $\Sigma(W)$ are in $\CD'$, and where the sequence
$$\xymatrix{0\ar[r] & X' \ar[r]^{f'} & Y \ar[r]^{g'} & Z' \ar[r] &0}$$
is exact in $\CD'$. Since $\CD$ is full in $\CC$, hence in $\CD'$ and
since any abelian category is idempotent complete, it follows that 
$X'$, $Z'$, $W$ belong to $\CD$ (up to isomorphism). Since $\CD$ is a 
full subcategory of $\CC$, hence of $\CD'$, it follows that the 
morphisms $f'$, $g'$ belong to $\CD$. Thus $f$ is  the direct sum in 
$\CD$ of $f'$ and the zero morphism $W\to$ $0$. But $f$ is also a 
monomorphism in $\CD$, and hence $W=0$. The result follows.
\end{proof}

\begin{Proposition} \label{abelianZorn}
Let $\CC$ be an essentially small triangulated category. Every 
distinguished abelian subcategory of $\CC$ is contained in a maximal 
distinguished abelian subcategory of $\CC$, with respect to the 
inclusion of subcategories.
\end{Proposition}

\begin{proof}
We may assume that $\CC$ is small, so that the distinguished abelian
subcategories form a set. Let $\CT$ be a totally ordered set of 
distinguished abelian subcategories of $\CC$, where the order is by 
inclusion. In view of Zorn's Lmma, we need to show that $\CT$ has an
upper bound. We claim that $\CE = \cup_{\CD\in\CT}\ \CD$ is such an 
upper bound. We need to show that $\CE$ is a distinguished abelian 
subcategory. By construction, $\CE$ is a full subcategory of $\CC$.
We show next that $\CE$ is an abelian category.
Let $f : X\to$ $Y$ be a morphism in $\CE$. Then there is $\CD\in$ $\CT$
containing $X$, $Y$, and since $\CD$ is a full subcategory of $\CC$,
it follows that $f$ is a morphism in $\CD$. Thus $f$ has a kernel
$a:W\to$ $X$ in $\CD$. We are going to show that $a$ is a kernel of
$f$ in $\CE$. Let $g : Z\to$ $X$ a morphism in $\CE$ such that
$f\circ g=0$. We need to show that $g$ factors uniquely through $a$.
Since $\CT$ is totally ordered, there is $\CD'\in$ $\CT$ such that
$\CD\subseteq$ $\CD'$ and such that $g$ is a morphism in $\CD'$.
By Proposition \ref{abelianinclusion}, the morphism $a$ remains a
kernel of $f$ as a morphism in $\CD'$. Thus there is a unique
morphism $h : Z\to$ $W$ in $\CD'$ such that $a\circ h=$ $g$. 
We need to show that $h$ is unique in $\CE$ with this property.
Let $j : Z\to$ $W$ be a morphism in $\CE$ such that $a\circ j=$ $g$.
Then $j$ belongs to a category $\CD''\in$ $\CT$, which we may choose
such that $\CD'\subseteq$ $\CD''$. Again by Proposition 
\ref{abelianinclusion}, the morphism $a$ remains a monomorphism in
$\CD''$. Since $a\circ j=$ $g=$ $a\circ h$, it follows that
$j=h$. This shows that the kernel of $f$ in any subcategory $\CD\in$
$\CT$ containing $f$ is the kernel of $f$ in $\CE$. A similar argument
shows that the cokernel of $f$ in any subcategory $\CD\in$ $\CT$
containing $f$ is the cokernel of $f$ in $\CE$. This implies
also that the canonical map $\coker(\ker(f))\to$ $\ker(\coker(f))$
in $\CE$ is an isomorphism, since it is an isomorphism in any
subcategory $\CD\in$ $\CT$ containing the morphism $f$.
By the above arguments,
any short exact sequence in $\CE$ is a short exact sequence in
$\CD$ for some $\CD\in$ $\CT$, hence can be completed to an exact
triangle in $\CC$. This shows that $\CE$ is a distinguished
abelian subcategory in $\CC$. Thus $\CT$ has an upper bound in the set 
of distinguished abelian subcategories of $\CC$. 
Zorn's Lemma implies the result.
\end{proof}

\begin{proof}[{Proof of Theorem \ref{ellA}}]
Statement (i) is Theorem \ref{AmodIembeddings}. For statement (ii), 
let $A$ be a finite-dimensional selfinjective algebra over a field $k$  
such that all simple $A$-modules are nonprojective. Then $J(A)$ contains 
its annihilator $\soc(A)$. Thus $\mod(A/J(A))$ is a distinguished 
abelian subcategory of $\modbar(A)$ containing all simple $A$-modules 
such that $\ell(\mod(A/J(A)))=$ $\ell(A)$. By Proposition 
\ref{abelianZorn} there is a maximal distinguished abelian subcategory 
in $\CC$ which contains $\mod(A/J(A))$. By Corollary 
\ref{AmodIembeddings-Cor1} the simple $A$-modules are exactly the simple 
objects in $\CD$. Thus $\ell(\CD)=$ $\ell(A)$, whence the result.
\end{proof}

The next two Propositions are tools for passing between short exact 
sequences in $\mod(A)$, for some finite-dimensional selfinjective 
algebra $A$, and short exact sequences in a distinguished abelian 
subcategory of the stable category $\modbar(A)$.

\begin{Proposition} \label{exactAandD1}
Let $A$ be a finite-dimensional selfinjective algebra over a field
$k$, and let $\CD$ be a distinguished abelian subcategory of
$\modbar(A)$. Let
$$\xymatrix{0 \ar[r] & X \ar[r]^{f} & Y  \ar[r]^{g} 
& Z \ar[r] & 0}$$
be a short exact sequence in $\CD$. Suppose that $X$, $Y$, $Z$
have no nonzero projective direct summands as $A$-modules.
The following hold.

\begin{enumerate}
\item[\rm (i)]
There is a finitely generated projective $A$-module $Q$ and a short 
exact sequence of $A$-modules
$$\xymatrix{0 \ar[r] & X \ar[r]^{a} & Y\oplus Q  \ar[r]^{b} 
& Z \ar[r] & 0}$$
such that $f$ and $g$ are the images of $a$ and $b$ in $\modbar(A)$,
respectively.

\item[\rm (ii)]
In addition, if $X$ or $Z$ is simple as an $A$-module, then $Q=$
$\{0\}$ in the first statement.
\end{enumerate}
\end{Proposition}

\begin{proof}
Since $\CD$ is a distinguished abelian subcategory of $\modbar(A)$,
it follows that the given exact sequence in $\CD$ gives rise
to an exact triangle 
$$\xymatrix{X \ar[r]^{f} & Y \ar[r]^{g} & Z \ar[r]^{h} & \Sigma(X)}$$
in $\modbar(A)$, for some morphism $h$. By the construction of
exact triangles in $\modbar(A)$, this exact triangle is induced
by a short exact sequence of $A$-modules of the form 
$$\xymatrix{0 \ar[r] & X\oplus P \ar[r]^{a} & Y\oplus Q  \ar[r]^{b} 
& Z \oplus R\ar[r] & 0}$$
for some finitely generated projective $A$-modules $P$, $Q$, $R$, such 
that $f$ and $g$ are the images of $a$ and $b$ in $\modbar(A)$. Since 
$a$ is injective, and since the $A$-module $P$ is projective, hence 
also injective, it follows that $a(P)\cong$ $P$ splits off the middle 
term $Y\oplus Q$. Since $Y$ has no nonzero projective summand, it 
follows that we may assume $P=$ $\{0\}$. A similar argument shows that 
we may assume $R=$ $\{0\}$, whence the first statement. 

For the second statement, assume first that $Z$ is simple as an 
$A$-module. Write $b=$ 
$\begin{pmatrix} r \\ s \end{pmatrix}$, 
where $r=$ $b|_Y : Y\to Z$
and $s=b|_Q : Q\to Z$. Since $g\neq 0$ in $\modbar(A)$, it follows 
that $r\neq 0$, hence $r$ is surjective as $Z$ is simple. Since
$Q$ is projective as an $A$-module, it follows that $s$ factors
through $r$. Write $s = r\circ t$ for some $t : Q\to$ $Y$. Set
$Q'=$ $\{(-t(x),x)\ |\ x\in Q\}$. This is a submodule of $Y\oplus Q$,
isomorphic to $Q$, and contained in $\ker(b)$. Since $Q'$ is
also injective, it follows that $Q'$ is isomorphic to a direct summand 
of $X$, hence $Q'=\{0\}$ by the first statement, and so also $Q=$ 
$\{0\}$. Assume next that $X$ is simple. Writing $a = (u,v) : X\to$ 
$Y\oplus Q$, we have that $u\neq$ $0$, so $u$ is injective. Thus $a(X)$ 
is not contained in the summand $Q$, hence intersects this summand 
trivially since $X$ is simple. Thus $b$ sends $Q$ to a submodule of $Z$ 
isomorphic to $Q$. Since $Q$ is also injective as an $A$-module, it 
follows that $Q$ is isomorphic to a direct summand of $Z$, hence zero 
by the first statement. This proves the second statement.
\end{proof}

\begin{Proposition} \label{exactAandD2}
Let $A$ be a finite-dimensional selfinjective algebra over a field
$k$, and let $\CD$ be a distinguished abelian subcategory of
$\modbar(A)$. Let
$$\xymatrix{0 \ar[r] & X \ar[r]^{a} & Y\oplus Q  \ar[r]^{b} 
& Z \ar[r] & 0}$$
be a short exact sequence of $A$-modules such that $Q$ is a projective
$A$-module, and such that the $A$-modules $X$, $Y$, $Z$ belong to $\CD$.
Suppose that as an object in $\CD$, $X$ has no nonzero injective direct 
summand, or that as an object in $\CD$, $Z$ has no nonzero projective 
direct sumand. Then the sequence
$$\xymatrix{0 \ar[r] & X \ar[r]^{f} & Y  \ar[r]^{g} 
& Z \ar[r] & 0}$$
is exact in $\CD$, where $f=\underline{a}$ and $g=$ $\underline{b}$
are the images in $\modbar(A)$ of $a$ and $b$, respectively.
\end{Proposition}
 
\begin{proof}
By the construction of $\modbar(A)$ as a triangulated category, 
the given short exact sequence of $A$-modules determines an exact 
triangle in $\modbar(A)$ of the form
$$\xymatrix{X \ar[r]^{f} & Y \ar[r]^{g} & Z \ar[r]^{h} & \Sigma(X)}$$
for some morphism $h$ in $\modbar(A)$. By Proposition \ref{DexactC}, 
this triangle is isomorphic to a direct sum of two exact triangles of 
the form
$$\xymatrix{X' \ar[r]^{f'} & Y \ar[r]^{g'} & Z' \ar[r]^{h'} 
& \Sigma(X')}$$
$$\xymatrix{W \ar[r] & 0 \ar[r] & \Sigma(W) \ar@{=}[r] & \Sigma(W)}$$
such that the sequence
$$\xymatrix{0\ar[r] & X' \ar[r]^{f'} & Y \ar[r]^{g'} & Z' \ar[r] &0}$$
is exact in $\CD$, where $W$ is an injective object in $\CD$ such that
$\Sigma(W)$ is a projective object in $\CD$. Thus $f$ is a monomorphism
in $\CD$ if and only if $W=\{0\}$, which is equivalent to 
$\Sigma(W)=\{0\}$, hence to $g$ being an epimorphism in $\CD$. 
The result follows.
\end{proof}

\begin{Remark} \label{mor-unique}
Any commutative rectangle in a triangulated category $\CC$ 
$$\xymatrix{ X \ar[r]^{f} \ar[d]_{a} & Y \ar[d]^{b} \\
            X' \ar[r]_{f'} & Y' } $$
can be completed to a morphism of exact triangles 
$$\xymatrix{X \ar[r]^{f} \ar[d]_{a} & Y \ar[r]^{g} \ar[d]^{b} 
& Z \ar[r]^{h} \ar[d]^{c}  & \Sigma(X) \ar[d]^{\Sigma(a)} \\
X' \ar[r]_{f'} & Y' \ar[r]_{g'} & Z' \ar[r]_{h'}  & \Sigma(X')}$$
In general, $c$ is not uniquely determined by $(a,b)$. 
If, however, the two exact triangles are determined by short exact 
sequences 
$$\xymatrix{0\ar[r] & X \ar[r]^{f} & Y \ar[r]^{g} & Z \ar[r] &0}$$
$$\xymatrix{0\ar[r] & X' \ar[r]^{f'} & Y' \ar[r]^{g'} & Z' \ar[r] &0}$$
in a distinguished abelian subcategory $\CD$ of $\CC$, then $a$ and $c$ 
are both determined by $b$ alone. Indeed, $\CD$ is a full subcategory of 
$\CC$, so $a$, $b$, $c$ all are morphisms in $\CD$, and since $f'$ is a 
monomorphism and $g$ an epimorphism in $\CD$, it follows that $b$ 
determines both $a$ and $c$. 
In particular, any endomorphism $(a,b,c)$ 
of the exact triangle 
$$\xymatrix{ X\ar[r]^{f} & Y \ar[r]^{g} & Z\ar[r]^{h} & \Sigma(Z) }$$
is determined by the endomorphism $b$ of $Y$, or equivalently, the 
algebra homomorphism from the endomorphism algebra of this triangle to 
$\End_\CC(Y)$ sending $(a,b,c)$ to $b$ is injective. This is a necessary 
criterion for an exact triangle to have the property that its components 
belong to a distinguished abelian subcategory of $\CC$.
\end{Remark}

\begin{Remark} \label{mor-unique2}
Remark \ref{mor-unique} can be rephrased as stating 
that the inclusion functor of a distinguished abelian 
subcategory $\CD$ of a triangulated category $\CC$ sends 
morphisms of exact sequences in $\CD$ to morphisms of exact 
triangles in $\CC$. Indeed,  if 
$$\xymatrix{0 \ar[r] & X \ar[d]_{a} \ar[r]^{f} & Y \ar[d]_{b} \ar[r]^{g} 
& Z \ar[d]^{c} \ar[r] & 0\\
0 \ar[r] & X' \ar[r]_{f'} & Y' \ar[r]_{g'} & Z' \ar[r] & 0}$$
is a commutative exact diagram in $\CD$, then $c$ is uniquely
determined by $b$, and hence the diagram 
$$\xymatrix{X \ar[d]_{a} \ar[r]^{f} & Y \ar[d]_{b} \ar[r]^{g} 
& Z \ar[d]^{c} \ar[r]^{h} & \Sigma(X) \ar[d]^{\Sigma(a)} \\
X' \ar[r]_{f'} & Y' \ar[r]_{g'} & Z' \ar[r]_{h'} & \Sigma(X')}$$
in $\CC$ is commutative, where $h$, $h'$ are the unique 
morphisms such that the rows are exact triangles. 

In a similar vein, given two composable monomorphisms
$X\to Y$ and $Y\to Z$ in $\CD$, the obvious diagram in $\CD$
$$\xymatrix{ X \ar[r] \ar@{=}[d] & Y \ar[r] \ar[d] & Y/X \ar[d] \\
X \ar[r] & Z \ar[r] \ar[d] & Z/X \ar[d] \\
 & Z/Y \ar@{=}[r]  & Z/Y }$$
describing the third isomorphism theorem can be extended 
uniquely to an octahedral diagram in $\CC$ of the form
$$\xymatrix{ X \ar[r] \ar@{=}[d] & Y \ar[r] \ar[d] & Y/X \ar[r] \ar[d] 
& \Sigma(X) \ar@{=}[d] \\
X \ar[r] & Z \ar[r] \ar[d] & Z/X \ar[d] \ar[r] & \Sigma(X) \ar[d] \\
 & Z/Y \ar@{=}[r] \ar[d]  & Z/Y \ar[r] \ar[d] & \Sigma(Y) \\
 & \Sigma(Y) \ar[r] & \Sigma(Y/X) 
}$$
\end{Remark}

The kernel and cokernel of a morphism in a distinguished abelian 
subcategory are related with the third term of the exact triangle
determined by that morphism, via the octahedron in $\CC$ associated
with an epi-mono factorisation of the morphism.

\begin{Proposition} \label{DoctahedronC}
Let $\CD$ be a distinguished abelian subcategory in a triangulated
category $\CC$. Every morphism $f : X\to Y$ in $\CD$ gives rise to an
octahedron in $\CC$ of the form
$$\xymatrix{ X \ar[r]^{u} \ar@{=}[d]&  C \ar[r]^{} \ar[d]_{v} 
&  \Sigma(\ker(f)) \ar[r] \ar[d]^{} &  \Sigma(X)\ar@{=}[d] \\
X \ar[r]_{f} &  Y \ar[r]_{g} \ar[d] &V \ar[r]_{h} \ar[d]_{}
 & \Sigma(X)\ar[d]^{\Sigma(u)} \\
 & \coker(f)  \ar@{=}[r] \ar[d]_{w} &  \coker(f) \ar[r]_{w}\ar[d] 
& \Sigma(C) \\
 & \Sigma(C) \ar[r]_{} &  \Sigma^2(\ker(f)) & & }$$
where $C$ is an object in $\CD$, $u$ an epimorphism in $\CD$, $v$
a monomorphism in $\CD$, and where $\ker(f)$ and $\coker(f)$ denote the 
kernel and cokernel of $f$ in $\CD$, respectively.
\end{Proposition}

\begin{proof}
Since $\CD$ is abelian, we have a canonical isomorphism  $C=$ 
$\coker(\ker(f))\cong$ $\ker(\coker(f))$ in $\CD$. Since $\CD$ is a 
distinguished abelian subcategory, the obvious short exact sequences
$$\xymatrix{0\ar[r] & \ker(f) \ar[r]  & X \ar[r]^{u}  & C \ar[r] &0}$$
$$\xymatrix{0\ar[r] & C \ar[r]^{v}  & Y \ar[r]  & \coker(f) \ar[r] &0}$$
can be completed to exact triangles
$$\xymatrix{\ker(f) \ar[r]  & X \ar[r]^{u}  & C \ar[r] & 
\Sigma(\ker(f)) }$$
$$\xymatrix{C \ar[r]^{v}  & Y \ar[r]  & \coker(f) \ar[r]^{w} & 
\Sigma(C) }$$
Turning the first of these two exact triangles yields an exact 
triangle 
$$\xymatrix{X \ar[r]^{u}   & C \ar[r] & \Sigma(\ker(f)) \ar[r] 
& \Sigma(X) }$$
Thus an octahedron associated with the factorisation 
$$\xymatrix{ X \ar[rr]^{f} \ar[dr]_{u} &               & Y \\
                                       & C \ar[ur]_{v} &    }$$
has the form as stated.
\end{proof}

\begin{Remark} \label{octahedronRem}
With the notation of Proposition \ref{DoctahedronC}, the kernel and 
cokernel of $f$ and the factorisation of $f$ via $C$ are unique up to 
unique isomorphism. Once fixed, they determine the morphisms in the
top horizontal and left vertical exact triangle uniquely, by Lemma 
\ref{hunique}. 
\end{Remark}

\section{Extension closed distinguished abelian subcategories}
\label{ExtSection}

As mentioned in the introduction, unlike
hearts of $t$-structures, distinguished abelian subcategories 
in a triangulated category need not be extension closed - see
Proposition \ref{notextclosed} below. We develop criteria for a 
distinguished abelian subcategory $\CD$ to be extension closed in a 
triangulated category $\CC$. Since we will compare the shift functor
on $\CC$ to the shift operator on $\CD$, we  specify in this section
shift functors of triangulated categories.

Let $(\CC, \Sigma)$ be a $k$-linear triangulated 
category and let $\CD$ be a distinguished abelian subcategory of $\CC$ 
such that $\CD\cong$ $\mod(D)$ for some finite-dimensional 
$k$-algebra $D$. This hypothesis ensures that $\CD$ has enough 
injective objects. We are going to compare $\Ext_\CC^n(U,V)=$ 
$\Hom_\CC(U,\Sigma^n(V))$ and $\Ext_\CD^n(U,V)$, where $n\geq 0$.
For $n=0$ these two spaces are equal since $\CD$ is full in $\CC$.
We investigate the case $n=1$.
In order to calculate $\Ext^1_\CD$, we will make use of the usual
shift operator $\Sigma_\CD$ on $\CD$, defined as follows.
For each object $U$ in $\CD$ choose a (minimal) injective envelope 
$\iota_U : U\to$ $I_U$ in $\CD$, and set $\Sigma_\CD(U)=$ 
$\coker(\iota_U)$. 
That is, we have a short exact sequence in $\CD$ of the form
$$\xymatrix{0\ar[r] & U\ar[r]^{\iota_U} & I_U\ar[r] & \Sigma_\CD(U)
\ar[r] & 0}$$

\begin{Definition} \label{SignaDSigmaC}
Let $(\CC, \Sigma)$ be a $k$-linear triangulated 
category and let $\CD$ be a distinguished abelian subcategory of $\CC$ 
such that $\CD\cong$ $\mod(D)$ for some finite-dimensional 
$k$-algebra $D$. For each object $U$ in $\CD$ and each 
nonnegative integer $n$, define a morphism 
$$\sigma_{n,U} : \Sigma_\CD^n(U) \to \Sigma^n(U)$$
in $\CC$ inductively as follows. We set $\sigma_{0,U}=\Id_U$, assuming 
implicitly that $\Sigma^0_\CD$ (resp. $\Sigma^0$) is the identity 
operator (resp. identity functor) on $\CD$ (resp. $\CC$). We define 
$$\sigma_{1,U} : \Sigma_\CD(U) \to \Sigma(U)$$
as the unique morphism such that the triangle 
$$\xymatrix{U \ar[r]^{u}   & I_U \ar[r] 
& \Sigma_\CD(U) \ar[r]^{\sigma_{1,U}} & \Sigma(U) }$$
is exact. For $n\geq$ $2$, we define
$$\sigma_{n,U} = 
\Sigma(\sigma_{n-1,U})\circ \sigma_{1,\Sigma^{n-1}_\CD(U)}\ ,$$
where we identify $\Sigma_\CD\circ\Sigma_\CD^{n-1}=\Sigma_\CD^n$ and
$\Sigma\circ\Sigma^{n-1}=\Sigma^n$. 
\end{Definition}

In the situation above, $\Ext^1_\CD(-,-)$ is
a subbifunctor of $\Ext^1_\CC(-,-)$ restricted to $\CD$.

\begin{Theorem} \label{extDextC}
Let $(\CC, \Sigma)$ be a $k$-linear triangulated 
category and let $\CD$ be a distinguished abelian subcategory of $\CC$ 
such that $\CD\cong$ $\mod(D)$ for some finite-dimensional $k$-algebra 
$D$. For any two objects $U$, $V$ in $\CD$, the morphism
$\sigma_{1,V} : \Sigma_\CD(V)\to$ $\Sigma(V)$ induces an
injective map $\Ext^1_\CD(U,V) \to \Ext^1_\CC(U,V)$
which is natural in $U$ and $V$.  
\end{Theorem}

\begin{proof}
We start with the standard description of
calculating $\Ext_\CD^1(U,V)$ using an injective resolution of $V$
$$\xymatrix{0\ar[r] & V \ar[r]^{\iota_V} & I^0 \ar[r]^{\delta^0}
& I^1 \ar[r]^{\delta^1} & I^2\ar[r] &\cdots}$$
with notation chosen such that $I^0=$ $I_V$, and $\Im(\delta^0)=$
$\ker(\delta^1)=$ $\Sigma_\CD(V)$. By definition, $\Ext^1_\CD(U,V)$ is
the degree $1$ cohomology of the cochain complex obtained from applying
$\Hom_\CD(U,-)$ to the above injective resolution of $V$, of the form
$$\xymatrix{\Hom_\CD(U,I^0) \ar[r] & \Hom_\CD(U, I^1) \ar[r] &
\Hom_\CD(U,I^2) \ar[r] & \cdots}$$
where the first two maps are induced by composing with $\delta^0$ and
$\delta^1$, respectively. A morphism $\varphi : U\to I^1$ is in the
kernel of the second map if and only if it factors through
$\ker(\delta^1)=$ $\Im(\delta^0)=$ $\Sigma_\CD(V)$. Thus the kernel of 
the map
$$\xymatrix{\Hom_\CD(U, I^1) \ar[r] & \Hom_\CD(U,I^2)}$$
can be identified with $\Hom_\CD(U,\Sigma_\CD(V))$, and hence
$\Ext^1_\CD(U,V)$ is the cokernel of the map
$$\xymatrix{\Hom_\CD(U, I^0) \ar[r] & \Hom_\CD(U,\Sigma_\CD(V))}$$
induced by composition with $\delta^0$.

Applying the functor $\Hom_\CC(U,-)$ to the exact triangle 
$$\xymatrix{V \ar[r]^{u}   & I_V \ar[r] 
& \Sigma_\CD(V) \ar[r]^{\sigma_{1,V}} & \Sigma(V) }$$
yields an exact sequence
$$\xymatrix{\cdots \ar[r] & \Hom_\CD(U, I_V) \ar[r] 
&\Hom_\CD(U, \Sigma_\CD(V)) \ar[r] & \Hom_\CC(U,\Sigma(V)) \ar[r]
&\cdots }$$
Taking the quotient of the middle term by the image of the left term
yields a monomorphism $\Ext_\CD(U,V)\to$ $\Hom_\CC(U,\Sigma(V))=$
$\Ext_\CC^1(U,V)$ as stated. We need to show the naturality. Since
this map is defined by applying the functor $\Hom_\CC(U,-)$ to the
above diagram, and since the Yoneda embedding $U\mapsto$ $\Hom_\CC(U,-)$
is contravariantly functorial in $U$, it follows immediately that the
map $\Ext_\CD(U,V)\to$ $\Hom_\CC(U,\Sigma(V))=$ $\Ext^1_\CC(U,V)$ is
contravariantly functorial in $U$. 

In order to show functoriality in
$V$, let $\psi : V\to$ $W$ be a morphism in $\CD$. Then $\psi$
extends to a morphism $I_V\to$ $I_W$, and hence there is a commutative
diagram of exact triangles
$$\xymatrix{V\ar[r] \ar[d]_{\psi} & I_V\ar[r]\ar[d] 
&\Sigma_\CD(V) \ar[r]^{\sigma_{1,V}} \ar[d]^{\tau} 
& \Sigma(V) \ar[d]^{\Sigma(\psi)} \\
W\ar[r]  & I_W\ar[r] &\Sigma_\CD(W) \ar[r]_{\sigma_{1,W}} 
& \Sigma(W) }$$
The morphism $\tau$ depends on the choice of an extension of
$\psi$ to $I_V\to$ $I_W$. If $\tau$, $\tau'$ are two morphisms
making the above diagram commutative, then
$\sigma_{1,W}\circ(\tau-\tau')=0$, and hence $\tau-\tau'$ factors
through the morphism $I_W\to \Sigma_\CD(V)$ in the diagram.
Thus applying $\Hom_\CD(U,-)$ to $\tau-\tau'$ induces the zero
map $\Ext^1_\CD(U,V)\to$ $\Ext^1_\CD(U,W)$, showing the
functoriality in $V$. This proves the result.
\end{proof}

The following immediate consequence of Theorem \ref{extDextC}
shows that $\CD$-mutation pairs (cf. \cite[Definition 2.5]{IY}), with $\CD$ 
a distinguished abelian subcategory equivalent to the module category 
of a finite-dimensional algebra, arise only if $\CD$ is semisimple.

\begin{Corollary} \label{extDextCCor}
Let $(\CC, \Sigma)$ be a $k$-linear triangulated 
category and let $\CD$ be a distinguished abelian subcategory of $\CC$ 
such that $\CD\cong$ $\mod(D)$ for some finite-dimensional $k$-algebra 
$D$. Suppose that $\Ext^1_\CC(X,Y)=0$ for all $X$, $Y$ in $\CD$.
Then the $k$-algebra $D$ is semisimple.
\end{Corollary}

\begin{proof}
The hypotheses and Theorem \ref{extDextC} imply that $\Ext^1_D(U,V)=0$ for
any two finitely generated $D$-modules $U$, $V$, whence the result.
\end{proof}

See Dugas \cite{Dugas15} for torsion pairs and mutation in stable module
categories of selfinjective algebras, as well as the references therein.
The next result is a  criterion when the canonical maps
$\Ext^1_\CD(U,V)\to$ $\Ext^1_\CC(U,V)$ in the previous Theorem yield an 
isomorphism of bifunctors on $\CD$. 

\begin{Theorem} \label{Dextclosed1}
Let $(\CC, \Sigma)$ be a $k$-linear triangulated 
category and let $\CD$ be a distinguished abelian subcategory of $\CC$ 
such that $\CD\cong$ $\mod(D)$ for some finite-dimensional $k$-algebra 
$D$. The category $\CD$ is extension closed in $\CC$ if and only if the 
morphisms $\sigma_{1,V}$ induce isomorphisms $\Ext^1_\CD(U,V)\cong$ 
$\Ext^1_\CC(U,V)$ for all objects $U$, $V$ in $\CD$.
\end{Theorem}

\begin{proof} 
We adjust the notation slightly in order to ensure that the terms in 
exact triangles in the proof below appear in alphabetical order; that 
is, we consider $\Ext^1_\CC(W,U)$ instead of $\Ext^1_\CC(U,V)$.
Suppose first that $\CD$ is extension closed in  $\CC$. Let $U$, $W$ be 
objects in $\CD$, and let $\psi : W\to$ $\Sigma(U)$ be a morphism in 
$\CC$; that is, $\psi\in$ $\Ext^1_\CC(W,U)$. It suffices to show that 
there exists a morphism $\varphi : W\to$ $\Sigma_\CD(U)$ in $\CC$ (which 
is then automatically in $\CD$ as $\CD$ is full, thus representing
an element in $\Ext^1_\CD(W,U)$) such that $\psi=$ 
$\sigma_{1,U}\circ\varphi$. Complete $\psi$ to an exact triangle in 
$\CC$ of the form
$$\xymatrix{U\ar[r] & V \ar[r] & W \ar[r]^{\psi} & \Sigma(U) }$$
Since $\CD$ is extension closed, it follows that $V$ can be chosen to
belong to $\CD$ (possibly after replacing $V$ by an isomorphic object).
The morphism $U\to$ $V$ belongs then to $\CD$, and although it need
not be a monomorphism in $\CD$, it is a direct sum of a monomorphism
and a zero morphism, by Proposition \ref{DexactC}. Thus the morphism
$\iota_U : U\to$ $I_U$ extends to a morphism $V\to$ $I_U$, and hence
there exists a morphism of exact triangles
$$\xymatrix{U\ar[r] \ar@{=}[d] & V \ar[r] \ar[d]^{\tau} 
& W \ar[r]^{\psi} \ar[d]^{\varphi} & \Sigma(U) \ar@{=}[d] \\
U\ar[r]_{\iota_U} & I_U \ar[r] & \Sigma_\CD(U) \ar[r]_{\sigma_{1,U}} 
& \Sigma(U) }$$
This shows that $\psi$ is the image of the class represented by 
$\varphi$ under the map $\Ext^1_\CD(W,U)\to$ $\Ext^1_\CC(W,U)$ induced
by composition with $\sigma_{1,U}$, and so the latter map is surjective.

Suppose conversely that the map $\Ext^1_\CD(W,U)\to$ $\Ext^1_\CC(W,U)$ 
is surjective, and let $\psi\in$ $\Ext^1_\CC(W,U)=$ 
$\Hom_\CC(W,\Sigma(U))$. Then there is $\varphi\in$ 
$\Hom_\CD(W, \Sigma_\CD(U))$ such that the rectangle
$$\xymatrix{W \ar[r]^{\psi} \ar[d]^{\varphi} & \Sigma(U) \ar@{=}[d] \\
\Sigma_\CD(U) \ar[r]^{\sigma_{1,U}} 
& \Sigma(U) }$$
is commutative in $\CC$, hence can be completed to a morphism of 
triangles of the form
$$\xymatrix{U\ar[r] \ar@{=}[d] & V \ar[r] \ar[d]^{\tau} 
& W \ar[r]^{\psi} \ar[d]^{\varphi} & \Sigma(U) \ar@{=}[d] \\
U\ar[r]_{\iota_U} & I_U \ar[r] & \Sigma_\CD(U) \ar[r]^{\sigma_{1,U}} 
& \Sigma(U) }$$
In order to show that $\CD$ is extension closed, we need to show that
$V$ is isomorphic to an object in $\CD$. Consider a pullback diagram in 
$\CD$ of the form
$$\xymatrix{V' \ar[r] \ar[d]_{\tau} & W \ar[d]^{\varphi} \\
I_U \ar[r] & \Sigma_\CD(U) }$$
Since the morphism $I_U\to$ $\Sigma_\CD(U)$ is an epimorphism
in $\CD$, so is the morphism $V'\to$ $W$ in the last rectangle.
Thus this rectangle can be completed to an exact commutative diagram in 
$\CD$ of the form
$$\xymatrix{0 \ar[r] & U\ar[r] \ar@{=}[d] & V' \ar[r] \ar[d]_{\tau} 
& W \ar[d]^{\varphi} \ar[r] & 0\\
0 \ar[r] & U \ar[r]_{\iota_U} & I_U \ar[r] & \Sigma_\CD(U) \ar[r] & 0}$$
This in turn can be completed to morphism of exact triangles
$$\xymatrix{U\ar[r] \ar@{=}[d] & V' \ar[r] \ar[d]^{\tau'} 
& W \ar[r]^{\psi'} \ar[d]^{\varphi} & \Sigma(U) \ar@{=}[d] \\
U\ar[r]_{\iota_U} & I_U \ar[r] & \Sigma_\CD(U) \ar[r]_{\sigma_{1,U}} 
& \Sigma(U) }$$
But then $\psi'=$ $\sigma_{1,U}\circ \varphi =$ $\psi$, and this 
forces $V'\cong$ $V$ in $\CC$.
This completes the proof.
\end{proof}

If $I$ is an injective module over a finite-dimensional $k$-algebra
$D$, then $\Ext^1_D(U,I)=$ $\{0\}$ for any $D$-module $U$.
Therefore, if a triangulated category $\CC$ has an extension closed
distinguished abelian subcategory equivalent to $\mod(D)$, then
$\Ext^1_\CC(U,I)$ must also vanish thanks to the previous Theorem
(where we identify $U$, $I$ to their images in $\CC$). It turns out
This yields the following characterisation of extension closed
distinguished abelian subcategories which are equivalent to $\mod(D)$.

\begin{Theorem} \label{Dextclosed2}
Let $(\CC, \Sigma)$ be a $k$-linear triangulated 
category and let $\CD$ be a distinguished abelian subcategory of $\CC$ 
such that $\CD\cong$ $\mod(D)$ for some finite-dimensional 
$k$-algebra $D$. Then $\CD$ is extension closed if and only if for
any two objects $U$, $Y$ in $\CD$ such that $Y$ is injective in $\CD$
we have $\Ext^1_\CC(U,Y)=$ $\{0\}$.
\end{Theorem}

\begin{proof}
Suppose that $\CD$ is extension closed. Let $U$, $Y$ be objects in 
$\CD$ such that $Y$ is injective in $\CD$. Then $\Ext^1_\CD(U,Y)=$ 
$\{0\}$. Theorem \ref{Dextclosed1} implies $\Ext^1_\CC(U,Y)=$ $\{0\}$.

Conversely, suppose that $\Ext^1_\CC(U,Y)=$ $\{0\}$ for any two objects 
$U$, $Y$ in $\CD$ such that $Y$ is injective in $\CD$. By
Theorem \ref{Dextclosed1}, it suffices to show that there is an
isomorphism $\Ext^1_\CD(U,V)\cong$ $\Ext^1_\CC(U,V)$ induced by
$\sigma_{1,V}$, for any two objects $U$, $V$ in $\CD$.  Let $U$, $V$
be objects in $\CD$. Consider a short exact sequence in $\CD$ of the
form 
$$\xymatrix{0\ar[r] & V\ar[r]^{\iota_V} & Y \ar[r] & \Sigma_\CD(V)
\ar[r] & 0}$$
for some injective object $Y$ in $\CD$. Applying $\Hom_\CD(U,-)$
yields a long exact sequence of $\Ext_\CD$-spaces. Since 
$\Ext^1_\CD(U,Y)=$ $\{0\}$, this long exact sequence yields in 
particular an exact $4$-term sequence
$$\xymatrix{0\ar[r]&\Hom_\CD(U,V)\ar[r]&\Hom_\CD(U,Y)\ar[r]
&\Hom_\CD(U,\Sigma_\CD(V)) \ar[r]&\Ext_\CD^1(U,V) \ar[r] & 0}$$
Completing the previous short exact sequence to an exact triangle
$$\xymatrix{V\ar[r]^{\iota_V} & Y \ar[r] & \Sigma_\CD(V)
\ar[r]^{\sigma_{1,V}} & \Sigma(V)}$$
and applying the functor $\Hom_\CC(U,-)$ yields a long exact sequence
of $\Ext_\CC$-spaces. Since $\Ext^1_\CD(U,Y)=$ $\{0\}$ by the 
hypotheses, this long exact sequence yields in particular an exact 
sequence
$$\xymatrix{\Hom_\CC(U,V)\ar[r]&\Hom_\CC(U,Y)\ar[r]
&\Hom_\CC(U,\Sigma_\CD(V)) \ar[r]&\Ext_\CC^1(U,V) \ar[r] & 0}$$
The first three terms coincide with the first three nonzero terms in
the previous $4$-term exact sequence because $\CD$ is a full 
subcategory of $\CC$, and hence the same is true for
the fourth terms. By construction, the isomorphism $\Ext^1_\CD(U,V)\cong$
$\Ext^1_\CC(U,V)$ arising in this way is induced by $\sigma_{1,V}$.
The result follows from Theorem \ref{Dextclosed1}. 
\end{proof}

We compare $\Sigma_\CD$, $\Sigma_{\CE}$ for distinguished abelian
subcategories $\CD\subseteq$ $\CE$ such that $\CD$, $\CE$
are equivalent ot module categories. As before, injective envelopes
are understood to be minimal. 

\begin{Proposition} \label{Sigmacompare}
Let $(\CC, \Sigma)$ be a $k$-linear triangulated category, and let 
$\CD$, $\CE$ be distinguished abelian subcategories of $\CC$ such that 
$\CD\cong$ $\mod(D)$ and $\CE\cong$ $\mod(E)$ for some 
finite-dimensional $k$-algebras $D$, $E$. Suppose that $\CD\subseteq$ 
$\CE$ and that the simple objects of $\CD$ and $\CE$ coincide.
We have a  monomorphism $\Sigma_\CD(X) \to \Sigma_\CE(X) $ in 
$\CE$,  for any object $X$ in $\CD$. 
\end{Proposition}

\begin{proof}
Let $X\to I$ be an injective envelope of $X$ in $\CD$, and let
$I\to J$ be an injective envelope of $I$ in $\CE$.  By
Proposition \ref{abelianinclusion}, $\CD$ is an abelian subcategory of 
$\CE$, and hence $X\to I$ remains a monomorphism in $\CE$. Thus 
the composition $X\to I\to J$ is a monomorphism in $\CE$. We need to
show that this is an injective envelope of $X$ in $\CE$. Since
$\CE$ is a module category of a finite-dimensional algebra, we need to
show that every simple subobject $S\to J$ of $J$ in $\CE$ factors through
the map $X\to I$. Note that $S\to J$ factors through $I\to J$
because $J$ is an injective envelope of $I$ in $\CE$. Since $S$ is also
simple in $\CD$ and $X\to I$ an injective envelope of $X$ in $\CD$, 
it follows that this map factors indeed through $X\to I$.
Thus the monomorphism $I\to J$ induces the required monomorphism 
via  $\Sigma_\CD(X)\cong$ $I/X\to$ $J/X\cong\Sigma_\CE(X)$.
\end{proof}

We collect some  technicalities needed for the proof
of the last part of Theorem \ref{kGmodN}. 

\begin{Lemma} \label{notextclosed1}
Let $A$ be a finite-dimensional selfinjective $k$-algebra, and let
$I$ be a proper ideal in $A$. The following hold. 
\begin{enumerate}
\item[{\rm (i)}] We have $\Hom_A^\pr(I, A/I)=\{0\}$.
\item[{\rm (ii)}] We have $\Ext^1_A(A/I, A/I)\cong\Hom_A(I,A/I)$.
\end{enumerate}
\end{Lemma}

\begin{proof}
Let $\tau : I\to$ $A/I$ be a homomorphism of left $A$-modules. Suppose
that $\tau$ factors through a projective $A$-module. Then $\tau$
factors through the canonical surjection $\pi : A\to$ $A/I$; that is, 
there is an $A$-homomorphism $\beta : I\to$ $A$ such that $\psi=$ 
$\pi\circ\beta$. Since $A$ is selfinjective, it follows that $\beta$
extends to an $A$-homomorphism $\alpha : A\to$ $A$, implying that
$\beta$ is induced by right multiplication with an element $x\in$ $A$.
Since $I$ is an ideal, it follows that $\Im(\beta)=$ $Ix\subseteq$ $I=$
$\ker(\pi)$, and hence that $\tau=$ $\pi\circ\beta=0$. This prove (i).
Note that $I=$ $\Omega_A(A/I)$, ignoring projective summands. Since $A$ 
is selfinjective, this implies that $\Ext^1_A(A/I,A/I)\cong$ 
$\Hombar_A(I,A/I)$. By (i) this is isomorphic to $\Hom_A(I,A/I)$, 
whence (ii).  
\end{proof}

The next Proposition shows that extension closed distinguished 
abelian subcategories in a stable module category $\modbar(A)$ of a 
finite-dimensional selfinjective algebra $A$ cannot have the form 
$\mod(A/I)$, where $I$ is an ideal contained in $J(A)$ such that $A/I$ 
is selfinjective. This proves Theorem \ref{AmodIembedding-intro} (ii).

\begin{Proposition} \label{notextclosed}
Let $A$ be a finite-dimensional selfinjective $k$-algebra, and let
$I$ be a nonzero proper ideal in $A$ such that $A/I$ is selfinjective.
The following hold. 

\begin{enumerate}
\item[{\rm (i)}] If $I\subseteq$ $J(A)$, then 
$\Ext^1_A(A/I,A/I) \neq$ $\{0\}$.
\item[{\rm (ii)}] If $r(I)\subseteq$ $I\subseteq$ $J(A)$, then 
$\mod(A/I)$ is a distinguished abelian subcategory of $\modbar(A)$
which is not extension closed in $\modbar(A)$.
\end{enumerate}
\end{Proposition}

\begin{proof}
Since $A/I$ is selfinjective, the assumption $I\subseteq$ $J(A)$
implies that every simple $A$-module is isomorphic to a submodule
of $A/I$. In particular, every simple $A$-module quotient of the 
nonzero ideal $I$ is isomorphic to an $A$-submodule of $A/I$, hence 
yields a nonzero $A$-homomorphism $I\to$ $A/I$. Thus (i) follows 
from Lemma \ref{notextclosed1} (ii). 
If $r(I)\subseteq$ $I$, then $\mod(A/I)$ is a distinguished 
abelian subcategory of $\modbar(A)$ by Theorem \ref{AmodIembedding1}.
If also $I\subseteq$ $J(A)$, then $\Ext^1_A(A/I,A/I) \neq$ $\{0\}$
by (i). Since $A/I$ is selfinjective, it follows from 
Theorem \ref{Dextclosed2} that $\mod(A/I)$ is not extension
closed in $\modbar(A)$
\end{proof}

\begin{Remark} \label{notextclosedRemark}
Let $A$ be a finite-dimensional $k$-algebra, and let
$I$ be a nonzero proper ideal in $A$ such that $A/I$ is selfinjective.
Then $\Ext^1_A(A/I,A/I)=\{0\}$ if and only if $\Ext^1_A(U,A/I)=$
$\{0\}$ for every finitely generated $A/I$-module. Indeed,
if $\Ext^1_A(A/I,A/I)=\{0\}$, then $\Ext^1_A(Y,Y)=$ $\{0\}$ for
any finitely generated projective $A/I$-module $Y$. Applying the
functor $\Hom_A(-,U)$ to a short exact sequence of the form
$$\xymatrix{0\ar[r] & V \ar[r] & Y \ar[r] & U \ar[r] & 0}$$
for some free $A/I$-module $Y$ of finite rank yields a long exact sequence
starting
$$\xymatrix{0\ar[r] & \Hom_{A/I}(U,Y) \ar[r] & \Hom_{A/I}(Y,Y) \ar[r]
& \Hom_{A/I}(V,Y)\ar[r] & \Ext^1_A(U,Y) \ar[r] & 0}\ .$$
Since $Y$ is also injective as an $A/I$-module, the map 
$\Hom_{A/I}(Y,Y)\to$ $\Hom_{A/I}(V,Y)$ is surjective. Thus
$\Ext^1_A(U,Y)=$ $0$, and hence $\Ext^1_A(U,A/I)=$ $\{0\}$.
The converse is obvious.
\end{Remark}

We will need the following well known fact; we
sketch a proof  for convenience.

\begin{Lemma} \label{H1specialcase}
Let $H$ be a finite group. Denote by $I(kH)$ the augmentation ideal
of $kH$, regarded as left $kH$-module,  and by $\Hom(H,k)$ the 
$k$-vector space  of group homomorphisms from $H$ to the 
additive group ofg $k$. We have a canonical $k$-linear isomorphism
$$\Hom_{kH}(I(kH), k) \cong \Hom(H,k)$$
sending a $kH$-homomorphism $\alpha : I(kH)\to$ $k$ to the
group homomorphism $\beta : H\to k$ defined by $\beta(y)=$ 
$\alpha(y-1)$ for all $y\in H$.
\end{Lemma}

\begin{proof} 
We check first that $\beta$ is a group homomorphism. Let $y$, $z\in$ $H$.
Then $yz-1=$ $y(z-1) + (y-1)$, hence $\beta(yz)=$ $y\alpha(z-1)+\alpha(y-1)=$
$\beta(y)+\beta(z)$, where we use that $y$ acts as identity on $k$. The
map $\alpha\mapsto\beta$ is therefore well-defined, $k$-linear, and
injective. For the surjectivity, let $\beta : H\to k$ be a group homomorphism.
We need to show that the $k$-linear map $\alpha : I(kH)\to$ $k$
sending $y-1$ to $\beta(y)$ is a $kH$-homomorphism. Let $y$, $z\in$ $H$.
We have $\alpha(z(y-1))=$ $\alpha(zy-1 - (z-1))=$ $\beta(zy)-\beta(z)=$
$\beta(zyz^{-1})=$ $\beta(y)=$ $\alpha(y-1)=$ $z\alpha(y-1)$, whence the
result.
\end{proof}

The following result proves the last statement in Theorem \ref{kGmodN}. 

\begin{Proposition} \label{notextclosedProp}
Let $k$ be a field of prime characteristic $p$, let $G$ be a finite
group, and let $N$ be a normal subgroup of order divisible by $p$ in 
$G$. If $O^p(N)$ is a proper subgroup of $N$, then the canonical image 
of $\mod(kG/N)$ in $\modbar(kG)$ is a distinguished abelian subcategory 
which is not extension closed. 
\end{Proposition}

\begin{proof}
The fact that $\mod(kG/N)$ embeds as a distinguished abelian
subcategory into $\modbar(kG)$ follows from Theorem \ref{kGexample}.
The kernel of the canonical algebra homomorphism $kG\to kG/N$ is
equal to $I=$ $kG\cdot I(kN)$, where $I(kN)$ is the augmentation
ideal of $kN$. Since $kG$ is free as a right $kN$-module, we have
$I\cong $ $kG\ten_{kN} I(kN)$ as left $kG$-modules. Thus
$$\Hom_{kG}(I, kG/I) \cong \Hom_{kG}(kG\ten_{kN} I(kN), kG/N)
\cong \Hom_{kN}(I(kN), kG/N)\ ,$$
where the last isomorphism uses Frobenius reciprocity. Since $N$
acts trivially on $kG/N$ on the left, it follows that the space
$\Hom_{kN}(I(kN), kG/N)$ is nonzero if and only if $\Hom_{kN}(I(kN),k)$
is nonzero. By Lemma \ref{H1specialcase} his space is equal to the space
$\Hom(N,k)$ of group homomorphisms from $N$ to the additive group $k$.
Since $k$ has characteristic $p$, this implies that $\Hom_k(N,k)$
is trivial if and only if $O^p(N)=N$. Thus if $O^p(N)$ is a proper
subgroup of $N$, then by Lemma \ref{notextclosed1}, the space
$\Ext^1_{kG}(kG/N, kG/N)$ is nonzero. 
It follows from Theorem \ref{Dextclosed2} that $\mod(kG/N)$ is not 
extension closed in $\modbar(kG)$.
\end{proof}

\begin{Corollary}\label{notextclosedCor}
Let $k$ be a field of prime characteristic $p$, let $G$ be a finite
group, and let $Q$ be a nontrivial normal $p$-subgroup of $G$.
Then the canonical image of $\mod(kG/Q)$ in $\modbar(kG)$ is a 
distinguished abelian subcategory which is not extension closed. 
\end{Corollary}

\begin{proof}
Since $O^p(Q)$ is trivial but $Q$ is not, this is a special case of
Proposition \ref{notextclosedProp}.
\end{proof}

\begin{Remark} 
A result in \cite{Dyer} gives a sufficient criterion when an 
extension closed exact subcategory $\CD$ of a triangulated category 
$\CC$ has the property that $\Ext^1_\CD$ and $\Ext^1_\CC$ are isomorphic 
as bifunctors on $\CD$.
\end{Remark}

\section{Selfinjective distinguished abelian subcategories}
\label{CabSection2}

Let $A$ be a finite-dimensional $k$-algebra and $Y$ a finitely 
generated $A$-module. We denote by $\mod_Y(A)$ the full $k$-linear 
subcateory of $\mod(A)$ of all $A$-modules which are isomorphic to 
$\Im(\varphi)$ for some $\varphi\in$ $\End_A(Y^m)$ and some positive 
integer $m$. We denote by $\add(Y)$ the full additive subcategory of 
$\mod(A)$ of modules which are isomorphic to finite direct sums of 
direct summands of $Y$. Clearly $\mod_Y(A)$ contains $\add(Y)$. 
By a result of Cabanes \cite[Theorem 2]{CabAst}, if $E=$ $\End_A(Y)$ is 
selfinjective, then the canonical functor $\Hom_A(Y,-) : \mod(A)\to$ 
$\mod(E^\op)$ restricts to a $k$-linear equivalence $\mod_Y(A)\cong$
$\mod(E^\op)$. We use the results and methods from Cabanes  \cite{CabAst}
to identify in a similar vein the distinguished abelian subcategories
constructed earlier in Theorem \ref{modDmodA}.
We denote by $\modbar_Y(A)$ the image of $\mod_Y(A)$ in $\modbar(A)$. If 
$Y$ has no nonzero projective direct summand, then no module in 
$\mod_Y(A)$ has a nontrivial projective summand, and hence the canonical 
functor $\mod_Y(A)\to$ $\modbar_Y(A)$ induces a bijection on isomorphism 
classes of objects. 

\begin{Theorem} \label{stableCabanes}
Let $A$ be a finite-dimensional self-injective $k$-algebra.
Let $Y$ be a finitely generated $A$-module such that the algebra 
$E=$ $\End_A(Y)$ is selfinjective. Suppose that $\End_A^\pr(Y)=$
$\{0\}$ and that that $Y$ is projective as an $E$-module. 
Set $D=$ $E^\op$.

\begin{enumerate}
\item[\rm (i)]
The functor $Y\ten_{D} - : \mod(D)\to$ $\mod(A)$ induces a full 
embedding $\Phi_Y : \mod(D)\to \modbar(A)$ of $\mod(D)$ as a 
distinguished abelian subcategory in $\modbar(A)$, and moreover 
$\Phi_Y$ induces an equivalence of abelian categories $\mod(D)\cong$ 
$\modbar_Y(A)$.
 
\item[\rm (ii)]
The $A$-module $Y$, regarded as an object in the abelian category 
$\modbar_Y(A)$, is a progenerator of $\modbar_Y(A)$.

\item[\rm (iii)]
The canonical functor $\mod(A)\to$ $\modbar(A)$ induces an isomorphism
of abelian categories $\mod_Y(A)\cong$ $\modbar_Y(A)$. 
\end{enumerate}
\end{Theorem}

As mentioned earlier, not every distinguished abelian subcategory of 
$\modbar(A)$ is of the form as described in Theorem \ref{stableCabanes}, 
since any selfequivalence of $\modbar(A)$ as a triangulated category 
induces a permutation on distinguished abelian subcategories which need 
not preserve the distinguished abelian subcategories of the form as 
described in Theorem \ref{stableCabanes}. Note that the hypothesis
$\End_A^\pr(Y)=\{0\}$ implies that $Y$ has no nonzero projective
direct summand. Therefore, the second statement of Theorem 
\ref{stableCabanes} implies that the isomorphism classes of 
indecomposable summands of the $A$-module $Y$ are 
determined by  $\modbar_Y(A)$.
 
\begin{Corollary} \label{Ydetermined}
Let $A$ be a finite-dimensional self-injective $k$-algebra.
Let $Y$, $Y'$ be finitely generated $A$-modules which both satisfy the 
hypotheses on $Y$ in Theorem \ref{stableCabanes}. Then $\modbar_Y(A)=$ 
$\modbar_{Y'}(A)$ if and only if $\add(Y)=$ $\add(Y')$ in $\mod(A)$. 
\end{Corollary}

\begin{Remark}
Corollary \ref{Ydetermined} does not imply that a distinguished abelian 
subcategory is necessarily determined by a progenerator - the same 
object in $\CC$ could be a progenerator of several distinguished abelian 
subcategories. This Corollary only asserts that the distinguished 
abelian subcategories obtained as in Theorem \ref{stableCabanes} are 
determined by their progenerators. See the Example \ref{progenExample}. 
\end{Remark}

We use the following notation from \cite{CabAst}. For any $k$-algebra
$A$, any two $A$-modules $Y$, $U$, and any subset $M$ of 
$\Hom_A(Y, U)$, we set $M\cdot Y=$ $\sum_{\mu\in M}\ \mu(Y)$; that is,
$M\cdot Y$ is the $A$-submodule of $U$ spanned by the sum of the 
images of the $A$-homomorphisms in $M$. Setting $E=$ $\End_A(Y)$, we 
consider $Y$ as an $A$-$E^\op$-bimodule in the obvious way. The 
following Proposition collects the technicalities for the proof of
Theorem \ref{stableCabanes}.

\begin{Proposition} \label{modYAabelian}
Let $A$ be a finite-dimensional self-injective $k$-algebra and $Y$ a 
finitely generated $A$-module such that $E=$ $\End_A(Y)$ is 
selfinjective. Suppose that $Y$ is projective as an $E$-module.
Set $D=$ $E^\op$.

\begin{enumerate}
\item[\rm (i)]
Let $n$ be a positive integer and let $M$ be an $E^\op$-submodule
of $\Hom_A(Y,Y^n)$. The canonical $A$-homomorphism
$$\Psi : Y\ten_{D} M \to M\cdot Y$$
sending $y\ten \mu$ to $\mu(y)$, where $y\in$ $Y$ and $\mu\in$ $M$,
is an isomorphism. In particular, $Y\ten_{D} M$ belongs to
$\mod_Y(A)$. 

\item[\rm (ii)]
Let $M$ be a finitely generated $D$-module. The canonical
$D$-homomorphism
$$M \to \Hom_A(Y,Y\ten_{D} M)$$
sending $m\in $ $M$ to the map $y\mapsto y\ten m$ for $y\in$ $Y$
is an isomorphism.

\item[\rm (iii)]
Let $U$ be an $A$-module contained in $\mod_Y(A)$. The canonical
evaluation map
$$\Phi : Y\ten_{D} \Hom_A(Y,U) \to U$$
sending $y\ten \eta$ to $\eta(y)$, where $y\in$ $Y$ and $\eta\in$
$\Hom_A(Y,U)$, is an isomorphism.

\item[\rm (iv)]
The category $\mod_Y(A)$ is an abelian subcategory of $\mod(A)$, and the 
functor $\Hom_A(Y,-) : \mod(A)\to \mod(D)$
restricts to an equivalence of abelian categories
$$\mod_Y(A)\cong \mod(D)\ $$
with an inverse induced by the functor 
$Y\ten_{D} - : \mod(D) \to \mod(A)$. 

\item[\rm (v)]
The equivalence $\mod_Y(A)\cong$ $\mod(D)$ in {\rm (iv)} sends
$Y$ to the regular $D$-module $D$. In particular, $Y$ is a
progenerator of the category $\mod_Y(A)$. 

\item[\rm (vi)]
If $\End_A^\pr(Y)=$ $\{0\}$, then $\Hom_A^\pr(U,V)=$ $\{0\}$ for any
two $A$-modules $U$, $V$ in $\mod_Y(A)$.
\end{enumerate}
\end{Proposition}

\begin{proof}
There is clearly a well-definded $A$-homomorphism
$\Psi : Y\ten_{D} M \to M\cdot Y$ as described in (i), and this map
is obviously surjective. To show that this map is also injective,
consider the diagram
$$\xymatrix{ Y\ten_{D} M \ar[r]^{\Psi}\ar[d] & M\cdot Y \ar[d] \\
Y\ten_{D} \Hom_A(Y,Y^n) \ar[r] & Y^n}$$
where the vertical maps are induced by the inclusions
$M\subseteq$ $\Hom_A(Y,Y^n)$ and $M\cdot Y\subseteq$ $Y^n$, and where 
the bottom horizontal
map is the obvious evaluation map.  A trivial verification shows that
this diagram commutes. Since $\Hom_A(Y,Y^n)$ is a free 
$D$-module of rank $n$, it follows that the bottom horizontal
map is an isomorphism. Since $Y$ is projective as a right
$D$-module by the assumptions, it follows that the left
vertical map is injective. This implies that $\Psi$ is injective,
whence (i). Since $E$, hence $D$, is self-injective, we may assume that 
the $D$-module $M$ in (ii) is a submodule of $\Hom_A(Y,Y^n)$ for 
some positive integer $n$. Applying the functor $\Hom_A(Y,-)$ to
the isomorphism $\Psi$ yields an isomorphism 
$$\Hom_A(Y, Y\ten_{D} M) \cong \Hom_A(Y, M\cdot Y)$$
By \cite[Lemma 5]{CabAst}, the right side in this isomorphism is
equal to $M$ (this is an equality of subsets of $\Hom_A(Y,Y^n)$).  
The inverse of this isomorphism is the map described in (ii).
For (iii), observe first that the map $\Phi$ is surjective, since $U$ 
is in $\mod_Y(A)$, hence a quotient of a finite direct sum of copies 
of $Y$. For the injectivity, again since $U$ is in $\mod_Y(A)$, hence 
isomorphic to a submodule of $Y^n$ for some positive integer $n$, it 
follows that there is a commutative diagram of the form
$$\xymatrix{ Y\ten_{D} \Hom_A(Y,U) \ar[rr]^{\Phi} \ar[d] 
& & U\ar[d] \\
Y\ten_{D} \Hom_A(Y,Y^n) \ar[rr] & & Y^n}$$
where the right vertical map is injective. The left vertical
map is then injective, too, since $Y$ is projective as an
$E^\op$-module, and the bottom horizontal map is an isomorphism.
This shows that $\Phi$ is injective, whence (iii). One can prove
(iii) also by applying (i) and (ii) with $U=$ $M\cdot Y$. 
Statement (iv) follows from (ii) and (iii).
Statement (v) is an immediate consequence of (iv). 
Statement (vi) follows from Lemma \ref{Homprzero}.
\end{proof}

\begin{proof}[Proof of Theorem \ref{stableCabanes}]
In the situation of Theorem \ref{stableCabanes},  Cabanes' linear 
equivalence from \cite[Theorem 2]{CabAst} is an equivalence of abelian 
categories, by Proposition \ref{modYAabelian} (iv). By construction,
$\modbar_Y(A)$ is the image in $\modbar(A)$ of $\mod_Y(A)$. Thus we 
need to show that the inclusion $\mod_Y(A)\subseteq$ $\mod(A)$ composed 
with the canonical functor $\mod(A)\to$ $\modbar(A)$ is still a full
embedding. By Proposition \ref{modYAabelian} (vi) we have
$\Hom_A^\pr(U,V)=$ $\{0\}$ for any two $A$-modules $U$, $V$ in 
$\mod_Y(A)$. This implies that the canonical functor $\mod_Y(A)\to$ 
$\modbar(A)$ is a full embedding.
Since exact triangles in $\modbar(A)$ are induced by short exact
sequences in $\mod(A)$, it follows that the image $\modbar_Y(A)$ of 
$\mod_Y(A)$ in $\modbar(A)$ is a distinguished abelian subcategory in 
$\modbar(A)$. Thus Theorem \ref{stableCabanes} follows from Proposition 
\ref{modYAabelian}.
\end{proof}

\begin{proof}[{Proof of Corollary \ref{Ydetermined}}]
Note that the hypothesis $\End_A^\pr(Y)=$ $\{0\}$ implies that $Y$ has 
no nonzero projective direct summand; similarly for $Y'$. Thus the 
linear subcategories $\add(Y)$ and $\add(Y')$ of $\mod(A)$ are equal if 
and only if their images in $\modbar(A)$ are equal. This equality is 
clearly equivalent to $\mod_Y(A)=$ $\mod_{Y'}(A)$, whence the result. 
\end{proof}

\section{Examples and further remarks}
\label{ExamplesSection}

The following example illustrates that Theorem \ref{modDmodA} and
Theorem \ref{stableCabanes} 
cover some cases not covered by Theorem \ref{AmodIembedding1}. 

\begin{Example} \label{hypothesesExample}
Let $A$ be a finite-dimensional selfinjective $k$-algebra. 
Let $Y$ be a nonprojective uniserial $A$-module of length $2$ with two 
non-isomorphic simple composition factors $S$ and $T$. Then 
$\End_A(Y)\cong$ $\Endbar_A(Y)\cong$ $k$, and 
hence $\mod_Y(A)=$ $\add(Y)$ is abelian semisimple, but is not the 
category of a quotient of $A$. Indeed, such a quotient algebra would have 
to be semisimple, but $\mod_Y(A)$ contains no simple $A$-module, because 
neither the simple quotient $S$ of $Y$ nor the simple submodule $T$ of 
$Y$ are contained in $\mod_Y(A)$.
\end{Example}

There are trivial examples of distinguished
abelian subcategories beyond those constructed in Theorem
\ref{modDmodA} and Theorem \ref{stableCabanes}.

\begin{Example}  \label{orthogonalExample}
Let $A$ be a split finite-dimensional selfinjective $k$-algebra. 
Let $n$ be a positive integer and let 
$\{X_i\ |\ 1\leq i\leq n\}$ be a set of $A$-modules which are pairwise 
orthogonal in $\modbar(A)$; that is, $\Endbar_A(X_i)\cong $ $k$ and 
$\Hombar_A(X_i,X_j)=$ $\{0\}$,  where $1\leq i,j\leq n$, $i\neq j$. Set
$Y=$ $\oplus_{i=1}^n X_i$. Then $\Endbar_A(Y)$ is a commutative
split semisimple $k$-algebra, and the image of $\add(Y)$ is a 
semisimple distinguished abelian subcategory of $\modbar(A)$,
equivalent to $\mod(\Endbar_A(Y))$. 
See for instance \cite{Pogorz1}, \cite{Pogorz2}, \cite{RiRo},
\cite{KoLiu}, \cite{Dugas15} for more details on orthogonal sets of 
modules in  $\modbar(A)$.
\end{Example} 

\begin{Example} \label{ellDinfinite-example}
Let $A$ be a finite-dimensional selfinjective $k$-algebra.
If $A$ has two nonisomorphic simple modules $S$, $T$ such that
$\dim_k(\Ext^1_A(S,T))\geq 2$ and if $k$ is infinite, then $A$ has 
infinitely many pairwise non-isomorphic uniserial modules of length $2$
with composition factors  $S$ and $T$, from top to bottom. The
$A$-endomorphism algebra of any such module is $1$-dimensional, and there 
is no nonzero $A$-homomorphism between any two non-isomorphic uniserial 
modules with these composition factors. Therefore the full additive 
subcategory of $\modbar(A)$  generated by these modules is a semisimple 
distinguished abelian subcategory with infinitely many isomorphism 
classes of simple objects. 
\end{Example}

The next example shows that the hypothesis on $\CD$ to contain all 
simple $A$-modules in the statement of Theorem \ref{AmodIembeddings} 
is necessary.

\begin{Example} \label{simpleExample}
 Suppose that $k$ is an algebraically closed field of characteristic $5$. 
Consider the algebra $A=$ $kD_{10}\cong$ $k(C_5\rtimes C_2)$. This is a
Nakayama algebra with two nonisomorphic simple modules $S$, $T$ and
uniserial projective indecomposable modules of length $5$. 
Let $U$ be a uniserial module of length $2$ with composition factors
$S$ and $T$. Then $\Endbar_A(U)=$ $\End_A(U)\cong$ $k$. Thus the
finite direct sums of modules isomorphic to $U$ form a semisimple
distinguished abelian subcategory $\CD$ of $\modbar(A)$ in which $U$ is
up to isomorphism the unique simple object. We have $\soc^2(A)=$
$r(J(A)^2)\subseteq$ $J(A)^2$. Thus $\mod(A/J(A)^2)$ is a 
distinguished abelian subcategory of $\modbar(A)$ containing the
simple $A$-modules $S$ and $T$ and all uniserial modules of length $2$. 
In particular, $\mod(A/J(A)^2)$ contains the subcategory $\CD$, but
the simple object $U$ in $\CD$ does not remain simple in 
$\mod(A/J(A)^2)$.
\end{Example}

Let $A$ be a symmetric $k$-algebra and $I$ a proper ideal in $A$. By a
result of Nakayama \cite[Theorem 13]{NakI}, the quotient algebra
$A/I$ is symmetric  if and only if $I=$ $\ann(z)$
for some $z\in$ $Z(A)$. In that case, if $s$ is a symmetrising form
on $A$, then $z\cdot s$ has kernel $I$ and induces a symmetrising 
form on $A/I$.  The following Proposition  shows that the elements 
$z\in$ $Z(A)$ satisfying $z^2=0$ parametrise the symmetric quotients 
$A/I$ of $A$ satisfying $\End_A^\pr(A/I)=$ $\{0\}$ through the 
correspondence $z\mapsto \ann(z)$.

\begin{Proposition} \label{z2Lemma}
Let $A$ be a symmetric $k$-algebra and $I$ a proper ideal in $A$ such
that $I=$ $\ann(z)$ for some $z\in$ $Z(A)$. We have $\End_A^\pr(A/I)=$ 
$\{0\}$ if and only if $z^2=0$.
\end{Proposition}

\begin{proof}
By the assumptions on $I$ and $z$, multiplication by $z$ induces an 
$A$-$A$-bimodule isomorphism $A/I\cong$ $Az$ mapping $a+I$ to $az$, 
where $a\in$ $A$. An endomorphism of $Az$ as a left $A$-module factors
through a projective module if and only if it factors through the
map $A\to$ $Az$ given by multiplication with $z$. Any $A$-homomorphism
$Az\to$ $A$ extends to an endomorphism of $A$ (because $A$ is symmetric)
hence is induced by right multiplication with an element $c$. Composing
the two maps $Az\to$ $A\to$ $Az$ given by right multiplication with $c$ 
and $z$, respectively, yields the endomorphism of $Az$ given by
right multiplication with $cz$, and the image of this endomorphism is
$Az^2c$. It follows that this endomorphism is zero for all $c\in$ $A$
if and only if $z^2=0$. The result follows.
\end{proof}

Different elements in $Z(A)$ may have the same annihilators. 
If $z$, $z'\in$ $Z(A)$ such that $I=$ $\ann(z)=$
$\ann(z')$, and if $s$ is a symmetrising form for $A$, then both
$z\cdot s$ and $z'\cdot s$ induce symmetrising forms on $A/I$. Thus 
there exists an element $y\in$ $A$ such that $y+I\in$ $Z(A/I)^\times$
and such that $z' = $ $yz$. Specialising Theorem \ref{AmodIembedding1}
to symmetric quotients of symmetric algebras yields the following result.

\begin{Proposition} \label{symmz2zero}
Let $A$ be an indecomposable nonsimple symmetric $k$-algebra. 
Let $z\in$ $Z(A)$ such that $z^2=0$ and such that $\soc^2(A)\subseteq$ 
$Az$. Set $I=$ $\ann(z)$.
Then $I\subseteq$ $J(A)^2$, the algebras $A$ and $A/I$ have the same
quiver, $A/I$ is symmetric, and the canonical functor $\mod(A/I)\to$ 
$\mod(A)$ induces an embedding $\mod(A/I)\to$ $\modbar(A)$ as 
distinguished abelian subcategory. In particular, $\modbar(A)$ has a 
connected distinguished abelian subcategory $\CD$ satisfying 
$\ell(\CD)=$ $\ell(A)$.
\end{Proposition}

\begin{proof}
By \cite[Theorem 13]{NakI} the algebra $A/I$ is symmetric. By the 
assumptions, we have $\soc^2(A)\subseteq$ $Az$. Taking annihilators yields
$I\subseteq$ $J(A)^2$, and hence $A$ and $A/I$ have the same quiver.
In particular, $\ell(A)=$ $\ell(A/I)$. Since $z^2=0$, it follows from 
Proposition \ref{z2Lemma} that  $\End_A^\pr(A/I)=$ $\{0\}$. Theorem 
\ref{AmodIembedding1} implies the result.
\end{proof}

\begin{Example} \label{qciEx}
Let $p$ be a prime number such that $p\geq 7$, and let $k$ be a field
of characteristic $p$. Let
$$A= k\langle x, y\ |\ x^p=y^p=0,\ xy=-yx\rangle$$
Then $A$ is a split local symmetric algebra of dimension $p^2$ with 
basis $\{x^iy^j\ |\ 0\leq i,j\leq p-1\}$. The linear map sending 
$x^{p-1}y^{p-1}$ to $1$ and all other monomials in this basis to $0$
is a symmetrising form for $A$. The monomials $x^iy^j$ with either
both $i$, $j$ even or one of them equal to $p-1$ form a basis of $Z(A)$.
The monomial $x^{p-1}y^{p-1}$ is a basis of $\soc(A)$, and the
set $\{x^{p-2}y^{p-1}, x^{p-1}y^{p-2}, x^{p-1}y^{p-1}\}$ is a basis of
$\soc^2(A)$. The element 
$z = x^{p-3}y^{p-3}$
belongs to $Z(A)$, satisfies $z^2=0$ (this is where we use $p\geq 7$), 
and we have $\soc^2(A)\subseteq$ $Az$. Thus $A$ and $z$ satisfy the 
assumptions in Proposition \ref{symmz2zero}. Therefore, setting $I=$ 
$\ann(z)$, we have a full embedding $\mod(A/I)\to$ $\modbar(A)$ as 
distinguished abelian subcategory, and the algebras $A$ and $A/I$ are
both symmetric and have the same quiver. More precisely, we have
$$A/I \cong k\langle x, y\ | x^3=y^3=0,\ xy=-yx\rangle\ ,$$
which is a $9$-dimensional quantum complete intersection, with
basis the image of the set of monomials $\{1, x, y, x^2, xy, y^2,
x^2y, xy^2, x^2y^2\}$.
Indeed, $z$ is annihilated, in $A$, by a monomial $x^iy^j$ if and only 
if at least one of $i$, $j$ is greater or equal to $3$.
\end{Example}

\begin{Remark} \label{center-Remark}
Let $\CC$ be an essentially small $k$-linear triangulated category, and
let $\CD$ be a $k$-linear distinguished abelian 
subcategory of $\CC$. Restriction to objects in $\CD$ induces a 
$k$-algebra homomorphism $Z(\CC) \to Z(\CD)$.
If $\CD\cong$ $\mod(D)$ for some $k$-algebra $D$, then this induces a
$k$-algebra homomorphism $Z(\CC) \to Z(D)$.
If in addition $D$ a finite-dimensional $k$-algebra, then this yields 
finite-dimensional $k$-algebra quotients of $Z(\CC)$. Finally, if
$\CC=$ $\modbar(A)$ for some finite-dimensional selfinjective
$k$-algebra $A$, then the canonical isomorphism  $Z(A)\cong$ 
$Z(\mod(A))$ induces an algebra homomorphism $\uZ(A)\to$ 
$Z(\modbar(A))$, where $\uZ(A)$ is the stable center of $A$. Thus 
restriction to a distinguished abelian subcategory of $\modbar(A)$ 
which is equivalent to $\mod(D)$ for some finite-dimensional $k$-algebra
$D$ yields a $k$-algebra homomorphism  $\uZ(A)\to$ $Z(D)$. Such
a homomorphism is in general neither injective nor surjective. 
\end{Remark}

\begin{Remark} \label{GrothendieckgroupRemark}
Let $\CC$ be an essentially small  triangulated category, and let 
$\CD$ be a distinguished abelian  subcategory of $\CC$. The 
Grothendieck group $a(\CD)$ of $\CD$ is the 
abelian group generated by the isomorphism classes $[X]$ of objects $X$ 
in $\CD$ subject to the relations $[X]-[Y]+[Z]$ for any short exact sequence
$0\to X\to Y\to Z\to 0$ in $\CD$. 
The Grothendieck group $a(\CC)$ of $\CC$ is the abelian
group generated by the isomorphism classes $[X]$ of objects $X$ in $\CC$
subject to the relations $[X]-[Y]+[Z]$ for any exact triangle
$ X\to Y\to Z\to \Sigma(X)$ in $\CC$. Since short exact sequences
in $\CD$ can be completed to exact triangles in $\CD$, it follows that 
the inclusion $\CD\to \CC$ induces a canonical group homomorphism 
$a(\CD)\to a(\CC)$. 
In general, this group homomorphism need not be injective or surjective. 
If $\CC$ is monoidal and $\CD$ a monoidal distinguished abelian
subcategory such that tensor products with objects in $\CC$ and $\CD$
preserve exact triangles in $\CC$ and short exact sequences in $\CD$, 
respectively, then the canonical map $a(\CD)\to$ $a(\CC)$ is a ring 
homomorphism.
If $\CC=\modbar(A)$ for some finite-dimensional selfinjective $k$-algebra
$A$ and $\CD$ contains all simple $A$-modules, then the canonical
group homomorphism $a(\CD)\to$ $a(\CC)$ is surjective. Note that if
the Cartan matrix of $A$ is nonsingular, then $a(\CC)$ is finite, while if
$\CD\cong$ $\mod(D)$ for some finite-dimensional $k$-algebra, then
$a(\CD)$ is the free abelian group on the set of isomorphism classes of
simple $D$-modules.
\end{Remark}

\begin{Remark} \label{fun-finite}
Let $A$ be a finite-dimensional self-injective algebra over a field $k$
and let $I$ be an ideal in $A$ which contains its right annihilator 
$r(I)$. Then the distiguished abelian subcategory $\mod(A/I)$ is
functorially finite in $\modbar(A)$  (cf. \cite[\S 3]{AusSmal80}, 
\cite{AusSmal81}). 
Indeed, let $U$ be an $A$-module.
Then $U/IU$ and the annihilator $U_I$ of $I$ in $U$ are in $\mod(A/I)$.
The canonical map $U\to$ $U/IU$, regarded as a morphism in $\modbar(A)$,
is a left approximation of $U$, and the inclusion $U_I\to U$, again 
regarded as a morphism in $\modbar(A)$, is a right approximation of $U$.
(Of course, $\mod(A/I)$ is also functorially finite in $\mod(A)$, by the
same argument.) 
\end{Remark}

The Tensor-Hom adjunction induces a natural transformation of
bifunctors at the level of stable categories, but this
need not be an isomorphism (cf. Proposition \ref{no-adjoint}).

\begin{Lemma} \label{stable-adj-Lemma}
Let $A$ be a finite-dimensional selfinjective $k$-algebra,
and let $D$ be a finite-dimensional $k$-algebra. Let $Y$ be a
finitely generated $A$-$D$-bimodule, $U$  a finitely generated
$A$-module and $V$ a finitely generated $D$-module. 
The Tensor-Hom adjunction 
$$\Psi : \Hom_A(Y\ten_D V, U)\cong \Hom_D(V, \Hom_A(Y,U))$$
sends $\Hom^\pr_A(Y\ten_D V, U)$  to $\Hom_D(V, \Hom^\pr_A(Y,U))$
and induces a natural map
$$\uPsi : \Hombar_A(Y\ten_D V, U)\to \Hom_D(V, \Hombar_A(Y,U))\ .$$
\end{Lemma}

\begin{proof}
We need to show that $\Psi$ sends $\Hom^\pr_A(Y\ten_D V, U)$ 
to $\Hom_D(V, \Hom^\pr_A(Y,U))$.
Let $\pi : P\to$ $U$ be a projective cover of $U$. Then any
$A$-homomorphism ending at $U$ which factors through a projective
$A$-module factors through $\pi$. Let $\varphi : Y\ten_D V\to$ $U$
be an $A$-homomorphism which factors through $\pi$. That is, there
is an $A$-homomorphism $\alpha : Y\ten_D V\to$ $P$ such that
$\varphi=$ $\pi\circ\alpha$. For $v\in$ $V$, denote by $\alpha_v :
Y\to$ $P$ the $A$-homomorphism defined by $\alpha_v(y)=$ 
$\alpha(y\ten v)$, for all $y\in$ $Y$. Through the 
adjunction $\Psi$, the homomorphism $\varphi$ corresponds to the
map $v\mapsto (y\mapsto \varphi(y\ten v))$ in 
$\Hom_D(V, \Hom_A(Y,U))$. Now $\varphi(y\ten v)=$
$\pi(\alpha(y\ten v))=$ $\pi(\alpha_v(y))$, hence the map
$y\mapsto$ $(\varphi(y\ten v))$ is equal to $\pi\circ\alpha_v$,
hence belongs to $\Hom_A^\pr(Y,U)$. This shows that $\Psi$ induces
a map $\uPsi$ as stated
\end{proof}

\begin{Remark} \label{stable-adj-Remark}
The map $\uPsi$ in Lemma \ref{stable-adj-Lemma} need not be an 
isomorphism. Consider for instance the case
$D=A$ and $Y=A$ regarded as an $A$-$A$-bimodule. Then $\Phi$ is the
canonical functor $\mod(A)\to$ $\modbar(A)$. The map $\uPsi$ is
zero for all $U$, $V$ (because $Y$ is projective as a left $A$-module), 
but if $U=V$ is nonprojective, then the left side in the map $\uPsi$
is $\Endbar_A(U)$, hence nonzero. 
\end{Remark}

\begin{Remark} \label{AR-Remark}
Let $\CD$ be a distinguished abelian subcategory in a triangulated
category $\CC$, and let $0\to X\to Y\to Z\to 0$ be a short exact
sequence in $\CD$. If the associated exact triangle 
$X\to Y\to Z\to\Sigma(X)$ in $\CC$  is an Auslander-Reiten triangle, then
the short exact sequence above clearly is an Auslander-Reiten sequence in 
$\CD$. The converse need not hold.
\end{Remark}


\end{document}